\documentclass[reqno]{amsart} %
\usepackage{amssymb, mathdots}
\usepackage{mathabx}
\usepackage{mathrsfs}
\usepackage[utf8x]{inputenc}
\usepackage{amsmath,  graphicx}
\usepackage{diagrams}
\usepackage{tikz}
\usetikzlibrary{matrix}
\usepackage{wrapfig}  
\usepackage{enumerate}
\usepackage{url}
\DeclareSymbolFont{bbold}{U}{bbold}{m}{n}
\DeclareSymbolFontAlphabet{\mathbbold}{bbold}
\def\qmod#1#2{{\hbox{}^{\displaystyle{#1}}}\!\big/\!\hbox{}_{
\displaystyle{#2}}}

\def\resto#1#2{{
#1\hskip 0.4ex\vline_{\hskip 0.2ex\raisebox{-0,2ex}
{{${\scriptstyle #2}$}}}}}

 \def\psp#1#2%
  {\mathop{}%
   \mathopen{\vphantom{#2}}^{#1}%
   \kern-\scriptspace%
    \hskip -0.3mm{#2}} 
 \def\psb#1#2%
  {\mathop{}%
   \mathopen{\vphantom{#2}}_{#1}%
   \kern-\scriptspace%
    \hskip -0.0mm{#2}} 
 \def\pscr#1#2#3%
  {\mathop{}%
   \mathopen{\vphantom{#3}}^{#1}_{#2}%
   \kern-\scriptspace%
    \hskip -0.3mm{#3}} 
\def\C{{\mathbb C}}

\def\N{{\mathbb N}}

\def\P{{\mathbb P}}

\def\R{{\mathbb R}}

\def\Z{{\mathbb Z}}

\def\cringle{\mathaccent23}

\def\textmap#1{\mathop{\vbox{\ialign{
                                  ##\crcr
      ${\scriptstyle\hfil\;\;#1\;\;\hfil}$\crcr
      \noalign{\kern 1pt\nointerlineskip}
      \rightarrowfill\crcr}}\;}}
\def\bigtextmap#1{\mathop{\vbox{\ialign{
                                  ##\crcr
      ${\hfil\;\;#1\;\;\hfil}$\crcr
      \noalign{\kern 1pt\nointerlineskip}
      \rightarrowfill\crcr}}\;}}
      
\newcommand{\cal}{\mathcal}
\def\textlmap#1{\mathop{\vbox{\ialign{
                                  ##\crcr
      ${\scriptstyle\hfil\;\;#1\;\;\hfil}$\crcr
      \noalign{\kern-1pt\nointerlineskip}
      \leftarrowfill\crcr}}\;}}
 
\def\ag{{\mathfrak a}}
\def\bg{{\mathfrak b}}
\def\cg{{\mathfrak c}}

\def\eg{{\mathfrak e}}

\def\lg{{\mathfrak l}}
\def\mg{{\mathfrak m}}

\def\pg{{\mathfrak p}}

\def\vg{{\mathfrak v}}

\def\Ag{{\mathfrak A}}
\def\Bg{{\mathfrak B}}
\def\Cg{{\mathfrak C}}

\def\Pg{{\mathfrak P}}

\def\Sg{{\mathfrak S}}
\def\Vg{{\mathfrak V}}

\def\Wg{{\mathfrak W}}
\def\Zg{{\mathfrak Z}}

\theoremstyle{remark}
\newtheorem{ex}{Example}[section]

\newtheorem{re}{Remark}

\theoremstyle{plain}

\newtheorem{sz}{Satz}[section]
\newtheorem{thry}[sz]{Theorem}
\newtheorem{pr}[sz]{Proposition}
\newtheorem{co}[sz]{Corollary}
\newtheorem{dt}[sz]{Definition}
\newtheorem{lm}[sz]{Lemma}

\def\tr{\mathrm {Tr}}

\def\SU{\mathrm {SU}}

\def\SL{\mathrm {SL}}

\def\gl{\mathrm {gl}}
\def\Pic{\mathrm {Pic}}

\def\NS{\mathrm{NS}}
\def\deg{\mathrm {deg}}
\def\Hom{\mathrm{Hom}}

\def\Tors{\mathrm{Tors}}

\def\vol{\mathrm{vol}}

\def\id{ \mathrm{id}}
\def\im{\mathrm{im}}
\def\rk{\mathrm {rk}}

\def\st{\mathrm {st}}
\def\ss{\mathrm {ss}}
\def\pst{\mathrm{pst}}
\def\sgs{\mathrm{prss}}

\def\coker{\mathrm{coker}}

\def\S{\mathrm{S}}
\def\ASD{\mathrm{ASD}}
\def\HE{\mathrm{HE}}

\def\supp{\mathrm{supp}}
\def\Supp{\mathrm{Supp}}
\def\tf{\mathrm{tf}}
\newcommand\smvee{{\hskip -0.1ex \raise 0.2ex\hbox{$\scriptscriptstyle\vee$}}}
\def\we{{\smvee \hskip -0.3ex \smvee}}

\newcommand{\longtwoheadrightarrow}{}% teste si deja defini
\DeclareRobustCommand{\longtwoheadrightarrow}{\relbar\joinrel\twoheadrightarrow}

\begin{document}

\title {A continuity theorem for families of sheaves on complex surfaces}
\author{Nicholas Buchdahl \and Andrei Teleman \and  Matei Toma}

 \address{Nicholas Buchdahl: 
Department of Mathematics, University of Adelaide, Adelaide, 5005 Australia, email: nicholas.buchdahl@adelaide.edu.au }
\address{Andrei Teleman: Aix Marseille Université, CNRS, Centrale Marseille, I2M, UMR 7373, 13453 Marseille, France, email: andrei.teleman@univ-amu.fr }
\address{Matei Toma:  Institut de Mathématiques Elie Cartan, Université de Lorraine, B. P. 70239, 54506 Vandoeuvre-lès-Nancy Cedex, France
email: matei.toma@univ-lorraine.fr}

\begin{abstract}
We prove that any flat family $(\mathcal{ F}_u)_{u\in U}$ of rank 2 torsion-free 
sheaves on a Gauduchon surface defines a continuous map on the semi-stable 
locus $U^{\mathrm {ss}}:=\{u\in U|\ \mathcal{ F}_u\hbox{ \sl is slope semi-stable\/}\}$ with
 values in the Donaldson-Uhlenbeck compactification of the corresponding 
instanton moduli space.

In the general (possibly non-Kählerian) case, 
the Donaldson-Uhlenbeck compactification is not a 
complex space, and the set $U^{\mathrm {ss}}$ 
can be a complicated subset of the base space $U$ that
is neither open or closed in the classical topology, 
nor locally closed in the Zariski topology.   
This result provides an efficient tool for the explicit description of 
Donaldson-Uhlenbeck compactifications on arbitrary Gauduchon surfaces.	
\end{abstract}


\maketitle

\tableofcontents

\section{Introduction}

The well-known Kobayashi-Hitchin correspondence (\cite{D}, \cite{Ko}, \cite{Bu1}, \cite{LY}, \cite{LT}) establishes an isomorphism ${\cal M}^\HE\to {\cal M}^\pst$ between the moduli space of Hermite-Einstein connections  on a Hermitian vector over a compact Gauduchon manifold \cite{Gau}, and  the corresponding moduli space of polystable holomorphic structures on the same bundle. Proved first by Donaldson for bundles on algebraic surfaces, and used as a tool for computing the Donaldson invariants of this special class of 4-manifolds,  this remarkable result and its subsequent generalisations have had an impressive impact in modern geometry.

%In this article we will concentrate on the general version of Donaldson's original  Kobayashi-Hitchin correspondence,   i.e. we will focus on  the correspondence between Hermite-Einstein connections  on   rank 2-bundles (which can be interpreted as $\PU(2)$-instantons),  and polystable rank 2-holomorphic over  Gauduchon surfaces (algebraic, or non-algebraic).   We shall explain the underlying motivation for this focus on Donaldson theory on non-Kählerian surfaces after the presentation of our main result.

This correspondence in particular gives a complex geometric interpretation (and implicitly a computation  method) for instanton moduli spaces. However, in general, an instanton moduli space  ${\cal M}^\ASD$ might be non-compact, and for effective applications, for instance for the construction   of Donaldson invariants \cite{DK}, one needs its canonical compactification $\overline{\cal M}^\ASD$, called the Donaldson-Uhlenbeck compactification. Therefore it is very important to  have a complex geometric interpretation, or at least a ``method of computation" in complex geometric terms, of these canonical compactifications.
\vspace{-2mm}\\

In the algebraic geometric framework such a complex geometric interpretation exists: one constructs a {\it continuous}, surjective comparison map ${\cal M}^\mathrm{G}\to \overline{\cal M}^\ASD$ from a moduli space  of torsion-free Gieseker semi-stable sheaves  to the Donaldson-Uhlenbeck compactification $\overline{\cal M}^\ASD$ (see  \cite{Li}, \cite{Morg}), the fibers of which can be described explicitly.  Jun Li's version of this important result goes further, and uses this comparison map to endow $\overline{\cal M}^\ASD$ with the structure of a projective algebraic variety.     

%What happens in the general case when the base 4-manifold is a compact complex surface? What happens in the non-Kählerian framework, i.e. on surfaces with  odd $b_1$?  

In the non-algebraic (even Kählerian) case such a comparison theorem is not available. Moreover, in the non-Kählerian   framework several unexpected, fundamental, surprising difficulties arise:
\begin{enumerate}
\item 	In many interesting cases, the Donaldson compactification $\overline{\cal M}^\ASD$ has no complex space structure. For example the compactified instanton moduli spaces used in  \cite{Te2} can be identified with  the 4-sphere $S^4$.
\item  The slope semi-stability condition is neither closed, nor open in families, not even with respect to the classical topology. Moreover, the slope semi-stability condition is not locally closed with respect to the Zariski topology. 
\item Whereas it is still true that any holomorphic family $({\cal F}_u)_{u\in U}$ of slope-semi-stable sheaves does define a continuous map $U\to \overline{\cal M}^\ASD$, one cannot obtain an open cover of $\overline{\cal M}^\ASD$ using images of such maps.
\item The Gieseker semi-stability and stability conditions for torsion-free sheaves can be naturally extended to the non-Kählerian framework \cite{BTT}, but neither  of them is an open condition, not even with respect to the classical topology. There does not exist a complex space classifying holomorphic families of Gieseker stable or Gieseker semi-stable torsion-free sheaves.       
\end{enumerate}

The dramatic consequences of these difficulties are:
\begin{enumerate}[1.]
\item Any attempt to prove a {\it general} theorem which endows the Donaldson-Uhlenbeck compactification $\overline{\cal M}^\ASD$ with a canonical complex space structure is hopeless. Note however that in \cite{BTT} we did endow a {\it special} class of   campactified instanton moduli spaces on Hopf surfaces  with complex space structures, but these structures are not obtained using holomorphic families of semi-stable sheaves.   
\item Any attempt to study the Donaldson-Uhlenbeck compactification $\overline{\cal M}^\ASD$ by comparing it with a moduli space of Gieseker semi-stable torsion-free sheaves cannot succeed.
 \item The maps defined by holomorphic families of slope semi-stable sheaves are not sufficient to understand the topology of the Donaldson-Uhlenbeck compactification $\overline{\cal M}^\ASD$.
\end{enumerate}

Our main result gives   a very general continuity theorem, which, in the projective algebraic framework specialises to the continuity of the comparison map ${\cal M}^G\to \overline{\cal M}^\ASD$ cited earlier.  Our theorem considers an {\it arbitrary} flat family $({\cal F}_u)_{u\in U}$ of rank 2 torsion-free sheaves on a Gauduchon surface, and gives a continuous map $U^\ss\to \overline{\cal M}^\ASD$  defined on the semi-stable locus
$$U^\ss:=\{u\in U|\ {\cal F}_u\hbox{ is slope semi-stable}\}
$$
(with respect to the relative topology of this subset). The point is that $U^\ss$ might be a rather complicated subset of the base space $U$. In general it is neither open or closed in the classical topology, nor locally closed in the Zariski topology. If one restricts the given family to $U^\ss$, the flatness condition will no longer make sense, because   $U^\ss$ is not necessarily  a complex space.

 Our result provides a very efficient tool for describing explicitly  Donaldson-Uhlenbeck compactifications $\overline{\cal M}^\ASD$ on arbitrary  Gauduchon surfaces, be  they algebraic or non-algebraic (including non-Kählerian). Indeed the family of continuous maps $U^\ss\to \overline{\cal M}^\ASD$ obtained in this way (defined on spaces $U^\ss$ {\it which in general are not complex spaces}) is sufficient to understand in detail the topology of $\overline{\cal M}^\ASD$ \cite{BTT}.\\

There are several important motivations underlying this study  of  Donaldson theory on general (not necessarily algebraic) Gauduchon surfaces:   

\begin{enumerate}[(a)]

\item  Recently Donaldson theory has found unexpected applications in non-Kählerian complex geometry \cite{Te1}, \cite{Te2}, proving results that do not seem to be accessible by any other means.

\item Moduli spaces of Hermite-Einstein connections on Gauduchon manifolds have an important differential geometric property, and intervene in modern differential geometric developments and in physics. More precisely, the stable part ${\cal M}^\st$ of the moduli space comes with a canonical Hermitian metric $H$ satisfying the strongly KT condition $\partial\bar \partial \omega_H=0$ \cite{LT}. It is  important to determine, in special cases, if this metric extends (possibly as a singular metric)  to the Donaldson-Uhlenbeck compactification.

\item   The continuity theorem  we prove in this article plays an important role in our new article \cite{BTT}, which gives an extension theorem for the complex structure on the regular part of certain instanton moduli spaces over class VII surfaces.
\end{enumerate}

Our method of proof is inspired by \cite{Morg}, but our result is  more general: whereas \cite{Morg} only considers families of   torsion-free Gieseker semi-stable sheaves with trivial determinant on a simply connected projective surface, our result concerns arbitrary families of   torsion-free sheaves on arbitrary Gauduchon surfaces. Moreover our proof fills in two apparent gaps in \cite{Morg}, as we   explain  in detail in   sections \ref{Pr1sect},  \ref{Pr2sect}.  We feel that filling in these gaps is necessary for such an important result.

The plan for the remainder of the paper is as follows: in section \ref{CTsection}  we construct the comparison map $U^\ss\to \overline {\cal M}^\ASD$ associated with a flat family of torsion-free, rank 2 sheaves, and  state the continuity theorem. The proof reduces to two convergence theorems: the first concerns    sequences of properly semi-stable sheaves, whereas the second deals with the stable case. These are treated respectively   in sections \ref{PropSSCase} and \ref{STCase}.

\section{The continuity theorem}\label{CTsection}

Let $(M,g)$ be a Riemannian, connected, oriented compact 4-manifold, $D$ be a Hermitian line bundle on $M$, and $a$ be a Hermitian connection on $D$. For a Hermitian   2-bundle $E$ on $M$ with $\det(E)=D$ put
$${\cal A}_a(E):=\{A\in {\cal A}(E)|\ \det(A)=a\},\ 
{\cal A}_a^{\ASD}(E):=\{A\in {\cal A}_a(E)|\  (F_A^0)^+=0\},$$
$$
\ {\cal M}^{\ASD}_a(E):=\qmod{{\cal A}_a^{\ASD}(E)}{{\cal G}_E},\
$$
where $F_A^0$ denotes the trace-free part of the curvature $F_A$ of $A$, and ${\cal G}_E:=\Gamma(M,\SU(E))$ is the $\SU(2)$-gauge group of $E$.

For $k\in\Z$, let $E_k$ be a Hermitian  2-bundle on $M$ with $\det(E_k)=D$, $c_2(E_k)=k$. Fix $c\in\Z$ and endow the union
$$\overline{{\cal M}}^{\ASD}_a(c)= \coprod_{\substack{4k-c_1^2(D)\geq 0\\ k\leq c}} {\cal M}^{\ASD}_a(E_{k})\times S^{c-k}(M) 
$$
with the Donaldson-Uhlenbeck topology \cite{DK}.\\

Let $X$ be a compact complex surface. We  identify the total symmetric power
$$S^*(X):=\coprod_{k\in\N} S^k(X)
$$
with the space of finite support maps $X\to\N$. Using this identification, for an element $\sigma\in S^*(X)$ we put 
$$|\sigma|=\sum_{x\in X}\sigma(x),$$
so one has $\sigma\in S^{|\sigma|}(X)$ for any $\sigma\in S^*(X)$. For a 0-dimensional torsion sheaf ${\cal Q}$ on $X$ we denote by $\Supp({\cal Q})$ the element of $S^*(X)$ defined by the map $X\ni x\mapsto \dim_\C({\cal Q}_x)$.  

Now let $(X,g)$ be a compact Gauduchon surface, $D$ be a Hermitian line bundle on $X$, and $E$ be a Hermitian  2-bundle on $X$ with $\det(E)=D$. For a fixed holomorphic structure ${\cal D}$ on $D$ we denote by ${\cal M}^\pst_{\cal D}(E)$ the moduli set of isomorphism classes of holomorphic structures ${\cal E}$ on $E$ with $\det({\cal E})={\cal D}$, modulo the $\SL(2)$ complex gauge group  ${\cal G}_E^\C:=\Gamma(M,\SL(E))$.  Putting $c:=c_2(E)$, note that ${\cal M}^\pst_{\cal D}(E)$ can be identified with the moduli space ${\cal M}^{\pst}_{[{\cal D}]}(c)$ of rank 2 holomorphic bundles ${\cal E}$ on $X$ with $c_2({\cal E})=c$ and $\det({\cal E})\simeq {\cal D}$. \\

The Kobayashi-Hitchin correspondence defines for any $k\in\Z$  a bijection 
$$\mathrm{kh}:{\cal M}^{\ASD}_a(E_{k})\textmap{\simeq}  {\cal M}^{\pst}_{[{\cal D}]}(k).$$ 

\begin{dt} Let ${\cal F}$ be a torsion-free coherent sheaf of rank 2 on $X$ whose bidual ${\cal F}^\we$ is slope semi-stable. We associate to ${\cal F}$ two invariants: the  {\rm bubbling invariant} $\sigma({\cal F})$ of ${\cal F}$, which is an element of $S^*(X)$, and the {\rm polystable class invariant} $P({\cal F})$ of ${\cal F}$ which is the isomorphism class of a rank 2-polystable bundle.  The two invariants are defined as follows:
\begin{itemize}
\item If ${\cal F}^\we$ is stable, put
$$\sigma({\cal F}):=\Supp\left(\qmod{{\cal F}^\we}{\cal F}\right),\ P({\cal F}):=[{\cal F}^\we].
$$
\item If ${\cal F}^\we$ is properly semi-stable, let
$$0\to {\cal L}\to {\cal F}^\we\to {\cal M}\otimes {\cal I}_Z\to 0
$$
be a destabilising short exact sequence, where ${\cal L}$, ${\cal M}$ are rank 1 locally free  sheaves, and $Z$ is a 0-dimensional locally complete intersection in $X$. Then put
$$\sigma({\cal F}):=\Supp\left(\qmod{{\cal F}^\we}{\cal F}\right)+\Supp({\cal O}_Z),\ P({\cal F}):=[{\cal L}\oplus {\cal M}].
$$
\end{itemize}
\end{dt}
\begin{re}
The two invariants are well defined.	
\end{re}
\begin{proof}
Indeed, let 	
$$0\to {\cal L}\textmap{j} {\cal F}^\we\textmap{p} {\cal M}\otimes {\cal I}_Z\to 0
$$
$$0\to {\cal L}'\textmap{j'}{\cal F}^\we\textmap{p'} {\cal M}'\otimes {\cal I}_{Z'}\to 0
$$
be two destabilising exact sequences of ${\cal F}^\we$, where ${\cal L}$, ${\cal L}'$, ${\cal M}$, ${\cal M}'$ are rank 1 locally free sheaves on $X$, and $Z$, $Z'$ are 0-dimensional locally complete intersections in $X$. If $p\circ j'\ne 0$, then $H^0({\cal L}'^\smvee\otimes {\cal M}\otimes {\cal I}_Z)\ne 0$. Since $\deg({\cal L}'^\smvee\otimes {\cal M})=0$, this implies that ${\cal L}'\simeq {\cal M}$, and $Z=\emptyset$, and the first exact sequence splits. Computing $\det({\cal F}^\we)$ and $c_2({\cal F}^\we)$ using the two exact sequences we get ${\cal M}'\simeq {\cal L}$ and $Z'=\emptyset$. Therefore, in this case, one gets the same results using the two exact sequences, namely $\sigma({\cal F})=\Supp\left(\qmod{{\cal F}^\we}{\cal F}\right)$, $P({\cal F})=[{\cal L}\oplus{\cal M}]=[{\cal L}'\oplus{\cal M}']$. 

If $p\circ j'= 0$, then $j'$ factorises through a sheaf monomorphism $k:{\cal L}'\to {\cal L}$. Since  ${\cal L}$, ${\cal L}'$ are locally free rank 1 sheaves with $\deg({\cal L}')=\deg({\cal L})$, it follows that this monomorphism is an isomorphism. 
%
%??? But  ${\cal L}\simeq {\cal L}'$ implies ${\cal M}\otimes {\cal I}_Z\simeq {\cal M}'\otimes {\cal I}_{Z'}$, hence ${\cal M}\simeq {\cal M}'$, and $Z=Z'$.  ????
%
Hence $j'=j\circ k$, ${\cal M}\otimes {\cal I}_Z\simeq {\cal M}'\otimes {\cal I}_{Z'}$, ${\cal M}\simeq {\cal M}'$, and $Z=Z'$.
\end{proof}

With these definitions we can state

\begin{thry}[The Continuity Theorem]\label{CT}  Let $(X,g)$ be a Gauduchon surface, ${\cal D}$ be a holomorphic line bundle on $X$, $c\in\Z$, $U$ be a reduced complex space, and $\mathscr{F}$ be a sheaf  of rank 2  on $U\times X$, flat over $U$, such that 
\begin{itemize}
\item  $\det({\cal F}_u)\simeq {\cal D}$ for any $u\in U$.
\item For any $u\in U$ the sheaf ${\cal F}_u$ defined by the restriction   $\resto{\mathscr{F}}{\{u\}\times X}$   is torsion-free with $c_2({\cal F}_u)=c$.
\end{itemize}
  Put
$$U^\ss:=\{u\in U|\ {\cal F}_u\hbox{ is slope semi-stable}\}.
$$ 
Then the map $\Phi_	{\mathscr{F}}:U^\ss\to \overline{{\cal M}}^{\ASD}_a(c)$ defined by
$$\Phi_{\mathscr{F}}(u):=\big(\mathrm{kh}^{-1}(P({\cal F}_u)),\sigma({\cal F}_u)\big)
$$
is continuous.
\end{thry}
\begin{proof}
Let $u\in U^\ss$. In order to prove the continuity of $\Phi_{\mathscr{F}}$ at $u$ it suffices to prove that  any sequence $(u_n)_n$ converging to $u$ in $U^\ss$ has  a subsequence 	$(u_{n_k})_k$   for which  $\lim_{k\to\infty}\Phi_{\mathscr{F}}(u_{n_k})=\Phi_{\mathscr{F}}(u)$. Writing $U^\ss=U^{\sgs}\cup U^{\st}$, where
\begin{align*}
U^{\sgs}&:=\{u\in U|\ {\cal F}_u \hbox{ is properly slope semi-stable}\},\\
U^{\st}&:=\{u\in U|\ {\cal F}_u \hbox{ is  slope  stable}\},	
\end{align*}
we see that it suffices to prove: 
\begin{enumerate}
\item Any sequence $(u_n)_n$ in $U^{\sgs}$ converging to $u$   has  a subsequence 	$(u_{n_k})_k$   for which  $\lim_{k\to\infty}\Phi_{\mathscr{F}}(u_{n_k})=\Phi_{\mathscr{F}}(u)$,  	
\item Any sequence $(u_n)_n$ in $U^{\st}$ converging to $u$   has  a subsequence 	$(u_{n_k})_k$   for which  $\lim_{k\to\infty}\Phi_{\mathscr{F}}(u_{n_k})=\Phi_{\mathscr{F}}(u)$.
\end{enumerate}
The first statement will follow from Propositions \ref{PropSS}, \ref{PropCont1}, whereas the second  is given by Proposition \ref{PropCont2}.
\end{proof}

\begin{re}
In the algebraic geometric framework semi-stability is a Zariski open condition. If $X$ is non-Kählerian, the set  $U^{\mathrm{ss}}$ might not be  even locally analytic. \end{re}
\begin{ex}
Let $(X,g)$ be a Gauduchon surface  with $b_1(X)=1$, $p_g(X)=0$. The Picard group $\Pic^0(X)$ of isomorphism classes of topologically trivial line bundles is isomorphic to $\C^*$. Let $\mathscr{L}$ be a Poincaré line bundle on $\Pic^0(X)\times X$. For any $\varepsilon>0$ the condition $|\deg_g(l)|<\varepsilon$ defines an annulus $A_\varepsilon\subset\Pic^0(X)$. For $l\in \Pic^0(X)$ denote by ${\cal L}_l$ the line bundle on $X$ associated with the restriction $\resto{\mathscr{L}}{\{l\}\times X}$.	Let $V\subset X$ be a Stein open subset, and $\varepsilon>0$ be sufficiently small. In \cite[Section 5]{BTT} we construct a rank 2 locally free sheaf $\mathscr{E}$ on $A_\varepsilon\times V\times X$, such for any $(l,x)\in  A_\varepsilon\times V$ the bundle ${\cal E}_{lx}$ on $X$ defined by the restriction $\resto{\mathscr{E}}{\{(l,x)\}\times X}$ is isomorphic to a Serre extension of the form
$$0\to {\cal L}_l\to {\cal E}_{lx}\to {\cal L}_l^\smvee\otimes {\cal I}_x\to 0.
$$
Putting $U:=A_\varepsilon\times V$, we see that in this case one has
$$U^{\ss}=A_\varepsilon^{\leq 0}\times V,
$$
where $A_\varepsilon ^{\leq 0}:=\{l\in A_\varepsilon|\ \deg_g(l)\leq 0\}$, so in this case $U^{\ss}$ is a manifold with boundary.
\end{ex}
\begin{ex} Keeping the notations from the previous example, let $\Pg^0$ be the quotient of $\Pic^0(X)$ by the involution $l\mapsto l^{\smvee}$, and $V\subset X$ be a contractible, Stein open subset. In \cite[Section 4.1]{BTT} we construct a sheaf $\mathscr{F}$ on $\Pg^0\times V\times X$, flat over $\Pg^0\times V$, such that for any $(p,x)\in \Pg^0\times V$ the sheaf ${\cal F}_{px}$ defined by the restriction $\resto{\mathscr{F}}{\{(p,x)\}\times X}$ can be written as a non-split extension
$$0\to {\cal L}_l\otimes {\cal I}_x\to {\cal F}_{px}\to {\cal L}_l^\smvee\to 0,
$$
for any representative $l\in p$. Putting $U:=\Pg^0\times V$, in this case we  have
$$U^\ss=\Cg^0\times V,
$$
where $\Cg^0\subset \Pg^0$ is the line segment obtained as the quotient of the circle %
$$C^0:=\{l\in\Pic^0(X)|\ \deg_g(l)=0\}$$
 by the involution $l\mapsto l^{\smvee}$. \end{ex}

\begin{re}\label{SmoothBase} Using the decomposition of $U$ as a union of its irreducible components, and a standard topological argument, it follows that it suffices to prove the theorem when $U$ is irreducible. Moreover, replacing $U$ by a desingularisation of $U$ and using the invariance of flatness under pull-back (see \cite[Remark p. 153]{Fi}) together with the surjectivity and properness of a desingularisation map, we see that it suffices to prove the statement under the assumption that $U$ is a connected smooth manifold. 
\end{re}

\section{Sequences of properly semi-stable sheaves}\label{PropSSCase}

\subsection{The difficulty}\label{Pr1sect}

With the assumptions and notations of Theorem \ref{CT}, let $(u_n)_n$ be a sequence in $U^\ss$ converging to $u\in U^\ss$.  In this section we prove that,  if the sheaves  ${\cal F}_{u_n}$ are properly semi-stable, then ${\cal F}_{u}$ is also properly semi-stable, and $\lim_{n\to \infty} \Phi_{\mathscr{F}} (u_n)= \Phi_{\mathscr{F}}(u)$. This part of the proof corresponds to section 4.1 in \cite{Morg}. Note however that there appears to be a significant gap in that section of \cite{Morg},  as we now explain. 

Since ${\cal F}_n:={\cal F}_{u_n}$ is properly slope semi-stable, we get, for any $n\in\N$, a slope-destabilising exact sequence of the form
$$0\to {\cal L}_n\otimes {\cal I}_{\Ag_n}\to {\cal F}_n\to {\cal M}_n\otimes  {\cal I}_{\Bg_n}\to 0,
$$ 
where ${\cal L}_n$, ${\cal M}_n$ are line bundles on $X$, and $\Ag_n$, $\Bg_n$ are 0-dimensional complex subspaces of $X$. The slope semi-stability of ${\cal F}_n$ implies $\deg_g({\cal L}_n)=\deg_g({\cal M}_n)$, but it does {\it not} imply $|\Ag_n|=|\Bg_n|$ as claimed in loc.\ cit.\ (even if one assumes that we are in the algebraic geometric framework, and even if all sheaves  are assumed to be Gieseker semi-stable). 

The argument  in \cite{Morg} uses  the equality $|\Ag_n|=|\Bg_n|$ in an essential way. Since one cannot use this equality, a substantial difficulty appears. As in loc.\ cit.\ one can prove that, passing to a subsequence if necessary, the sequences $(\Ag_n)_n$, $(\Bg_n)_n$, $[{\cal L}_n]_n$, $[{\cal M}_n]_n$ converge in the respective spaces to points,  $\Ag_\infty$, $\Bg_\infty$, $[{\cal L}_\infty]$, $[{\cal M}_\infty]$. Using a well-known semicontinuity theorem for the $\Hom(\cdot,\cdot)$ functor, one obtains a non-trivial sheaf morphism ${\cal L}_\infty\otimes  {\cal I}_{\Ag_\infty}\to {\cal F}_u$. 
However, the  cokernel of the obtained morphism ${\cal L}_\infty\otimes  {\cal I}_{\Ag_\infty}\to {\cal F}_u$ might contain torsion, so one must carefully control  the contribution of this torsion to $\Phi_{\mathscr{F}}(u)$.   Our proof is based on a  general ``compensation principle" for    families of rank 1 sheaves (Proposition \ref{newsigma}). This principle concerns flat families of rank 1 sheaves which might contain sheaves with non-trivial torsion, and states that the appearance of torsion in the limit sheaf (when passing to the limit in the parameter space) is compensated by a jump of the singularity set of its bidual. The main ingredients in the proof of this compensation principle are the Hironaka flattening theorem, and Barlet's theory of cycle spaces.

The claimed equality $\lim_{n\to \infty} \Phi_{\mathscr{F}} (u_n)= \Phi_{\mathscr{F}}(u)$ is proved in Proposition \ref{PropCont1}. For this result we make  use of deep complex-geometric results, for example of Flenner's representability theorem for  the functor $\underline{{\cal H}om}_S({\cal A},{\cal B})$ (see \cite[sec. 3.2]{Fl}) associated to a pair $({\cal A},{\cal B})$ of coherent sheaves defined on a space over a complex space $S$.  The main difficulty is to show that, after passing to a subsequence, the sequence of triples  $(\Ag_n,[{\cal L}_n],{\cal L}_n\otimes {\cal I}_{\Ag_n}\to {\cal F}_n )_n$ and also its limit $(\Ag_\infty,[{\cal L}_\infty],{\cal L}_\infty\otimes {\cal I}_{\Ag_\infty}\to {\cal F}_u )$ all belong to the same analytic family; we will then apply our compensation principle to the corresponding family of quotients.  Note that this part of the proof, in which we  pass from a sequence to an analytic family, illustrates  a significant general difficulty in proving comparison theorems between a complex-geometric and a gauge-theoretical moduli space: whereas the topology of a   gauge-theoretical moduli space is controlled by specifying the converging sequences,  the topology of a complex-geometric moduli space  is controlled using holomorphic families.

\subsection{The proof}

With the notations and assumptions of Theorem \ref{CT}, denote by $U^\sgs\subset U$ the subset 
$$U^{\sgs}:=\{u\in U|\ {\cal F}_u \hbox{ is properly slope semi-stable}\}.
$$
In this section we will prove that $U^\sgs$ is closed in $U^\ss$ and that the restriction $\resto{\Phi_{{\cal F}}}{U^\sgs}$ is continuous, which will complete the first part of the proof of Theorem \ref{CT}.

\begin{pr} \label{ExactSeq}
Let $X$ be a complex surface, ${\cal F}$ be a rank 2 torsion-free sheaf on $X$, and  ${\cal A}\subset {\cal F}$ be a rank 1 subsheaf such that the quotient ${\cal B}:={\cal F}/{\cal A}$ is also torsion-free. Put  
$${\cal N}:=\qmod{{\cal F}^\we/{\cal A}}{\Tors({\cal F}^\we/{\cal A})}.
$$
Then one has a canonical exact sequence of torsion sheaves on $X$:
$$0\to \qmod{{\cal A}^\we}{{\cal A}}\to \qmod{{\cal F}^\we}{{\cal F}}\to \qmod{{\cal B}^\we}{{\cal B}}\to \qmod{{\cal N}^\we}{{\cal N}}\to 0.
$$
\end{pr}
\begin{proof}
Put  ${\cal E}:={\cal F}^\we$, ${\cal Q}:={\cal E}/{\cal F}$, denote by $q:{\cal E}\to {\cal N}$ the canonical epimorphism, and note that the kernel ${\cal L}:=\ker(q)$ is reflexive and contains ${\cal A}$. Moreover, on $X\setminus\supp({\cal Q})$ the sheaf ${\cal E}$ coincides with ${\cal F}$, ${\cal E}/{\cal A}$ coincides with ${\cal B}$ (hence is torsion-free), ${\cal N}$ coincides with ${\cal B}$, and ${\cal L}$ coincides with ${\cal A}$. Since $\mathrm{codim}(\supp({\cal Q}))\geq 2$, it follows that ${\cal L}$ coincides with ${\cal A}^\we$. We get a commutative diagram with exact rows and exact columns:
$$
\begin{diagram}[s=6mm, w=6mm, midshaft]
&&0&& 0 &&0&&\\
&&\dTo&&\dTo &&\dTo&&\\
0&\rTo&{\cal A}&\rInto^{\alpha} &{\cal L}={\cal A}^\we&\rOnto&{\cal S}&\rTo&0\\
&&\dInto^{a}&&\dInto^{j} &&\dTo^{u}&&\\
0&\rTo &{\cal F}&\rInto^{\iota} &{\cal E} &\rOnto^{\pi} &{\cal Q} &\rTo &0\\
&&\dOnto^{b}&&\dOnto^{q} &&\dTo^{v}&&\\
0&\rTo&{\cal B}&\rTo^{\beta}&{\cal N} &\rOnto&{\cal T}&\rTo&0\\
&&\dTo&&\dTo &&\dTo&&\\
&&0&&0 &&0&&\\
\end{diagram}
$$
The morphism 
$$\beta:{\cal B}=\qmod{{\cal F}}{{\cal A}}\to \qmod{{\cal E}}{{\cal L}}={\cal N}$$
induced by the pair $(\alpha,\iota)$ is injective, because  its kernel, being supported on $\supp({\cal Q})$, is a torsion sheaf. On the other hand ${\cal B}$ is torsion-free, by hypothesis.  

Putting  ${\cal S}:=\coker(\alpha)$, ${\cal Q}:=\coker(\iota)$,
 ${\cal T}:=\coker(\beta)$, and denoting by $u$, $v$ the morphisms induced by $j$ and $q$, we see that the sequence
$$0\to {\cal S}\textmap{u} {\cal Q}\textmap{v} {\cal T}\to 0
$$
is exact, as claimed. Since ${\cal S}$ is a torsion sheaf, and ${\cal N}$ is torsion-free, we also have ${\cal S}={\cal T}ors({\cal E}/{\cal A})$. Note that the morphism $\tilde \beta:{\cal B}^\we\to{\cal N}^\we$ induced by $\beta$ is an isomorphism because $\beta$ is an isomorphism on $X\setminus\supp({\cal Q}$), and that the morphism $\bar \beta:\qmod{{\cal B}^\we}{{\cal B}}\to \qmod{{\cal N}^\we}{{\cal N}}$ induced by the pair $(\beta,\tilde\beta)$ is an epimorphism. Applying the snake lemma to the diagram
$$
\begin{diagram}[s=6mm, w=6mm, midshaft]
&&0 &&0 &&  &&\\
&&\dTo&&\dTo && &&\\
0 &\rTo& {\cal B}&\rTo & {\cal B}^\we&\rTo & \qmod{{\cal B}^\we}{{\cal B}}&\rTo& 0\\
&&\dTo^{\beta}&&\dTo^{\tilde\beta}_{\simeq}&&\dTo^{\bar\beta}\\
0&\rTo&{\cal N}&\rTo& {\cal N}^\we &\rTo&\qmod{{\cal N}^\we}{{\cal N}}&\rTo &0\\
&&\dTo&&\dTo&&\dTo&&\\
&&{\cal T}&&0&&0&&\\
&&\dTo&& && && &&\\
&&0
\end{diagram}
$$
we obtain an isomorphism ${\cal T}=\coker(\beta)\simeq \ker(\bar\beta)$, which proves the claim.
\end{proof}

\begin{co}\label{sigma}
In the conditions and with the notations of Proposition 	\ref{ExactSeq} suppose that ${\cal F}$ is properly  semi-stable,  that ${\cal A}$ is a destabilising rank 1 subsheaf of ${\cal F}$, and let ${\cal B}:=\qmod{{\cal F}}{{\cal A}}$.
\begin{enumerate}
\item If the quotient ${\cal B}$ is 	torsion-free, then one has
$$\sigma({\cal F})=\Supp({\cal A}^\we/{\cal A})+\Supp({\cal B}^\we/{\cal B}).
$$
\item If the torsion subsheaf ${\cal T}ors({\cal B})$ is non-trivial, then 
\begin{enumerate}
\item The support of ${\cal T}ors({\cal B})$  is 0-dimensional. 
\item One has 
$$\Supp({\cal T}ors({\cal B}))\leq \Supp\big({\cal A}^\we/{\cal A}\big).$$ 
\item Putting ${\cal B}_0:={\cal B}/{\cal T}ors({\cal B})$, one has
$$\sigma({\cal F})=\Supp({\cal A}^\we/{\cal A})+\Supp({\cal B}^\we_0/{\cal B}_0)-\Supp({\cal T}ors({\cal B})).
$$
\end{enumerate}
\end{enumerate}
 \end{co}
 \begin{proof}
 The first statement follows directly from 	Proposition \ref{ExactSeq} using the additivity  of $\Supp$ with respect to exact sequences of 0-dimensional sheaves. 
 
 To prove (2)(a), note that the  pre-image    of ${\cal T}ors({\cal B})$ under the quotient map ${\cal F} \to {\cal B}$ is  a rank 1-subsheaf ${\cal A}' \subset {\cal F}$ which fits in a short exact sequence
$$0\to {\cal A}\to {\cal A}'\to {\cal T}ors({\cal B})\to 0.
$$
This shows that $\deg_g({\cal A}')=\deg_g({\cal A})+\deg_g({\cal T}ors({\cal B}))$ with  $\deg_g({\cal T}ors({\cal B}))\geq 0$. Since $\deg_g({\cal A})=\frac{1}{2}\deg_g({\cal F})$, ${\cal F}$ is semi-stable and contains ${\cal A}'$, we get $\deg_g({\cal T}ors({\cal B}))=0$, which implies that the support of ${\cal T}ors({\cal B})$ is 0-dimensional.

For  (2)(b), using the snake lemma as in the proof of Proposition \ref{ExactSeq}, we obtain  an isomorphism
$$\ker \big(\qmod{{\cal A}^\we}{{\cal A}}\longtwoheadrightarrow  \qmod{{\cal A}'^\we}{{\cal A}'}\big)\simeq \coker({\cal A}\hookrightarrow {\cal A}')= {\cal T}ors({\cal B}),
$$
which gives $\Supp({\cal A}^\we/{{\cal A}})=\Supp({\cal A}'^\we/{{\cal A}'})+\Supp({\cal T}ors({\cal B}))$.

Finally (2)(c) follows from (1) applied to the subsheaf ${\cal A}'\subset {\cal F}$.

\end{proof}

Recall that a reflexive rank 1 sheaf on a smooth manifold is locally free. In particular, if $\mathscr{F}$ is torsion-free of rank 1 on $S\times X$ with $S$ smooth, we  obtain a sheaf monomorphism  $\mathscr{F}\hookrightarrow \mathscr{F}^\we$, where $\mathscr{F}^\we$ is locally free of rank 1. Therefore $\mathscr{F}$ can be identified with $\mathscr{F}^\we\otimes{\cal I}_Z$ for a complex subspace of $S\times X$ which has codimension $\geq 2$ at any point.  In particular $\supp(Z)$   coincides with {\it the bad set} $B(\mathscr{F})$ of $\mathscr{F}$ defined as the locus of points where ${\cal F}$ fails to be locally free \cite[section 4.4]{GR}.
   
 \begin{lm}\label{equiv} Let $X$ be a complex surface, $S$ be a complex manifold, and $\mathscr{F}$ be a torsion-free rank 1 coherent sheaf on $S\times X$, flat over $S$. Denote by ${\cal F}_s$ the restriction of $\mathscr{F}$ to $X_s:=\{s\}\times X$. Let $Z\subset S\times X$ be the complex subspace of $S\times X$   defined by the equality
 $$\mathscr{F}=\mathscr{F}^\we\otimes{\cal I}_Z. $$
  Then  
 \begin{enumerate}
 	
\item For any point $s\in S$  the sheaf ${\cal F}_s$ is a rank 1 sheaf on $X_s$.
 \item 	For any point $s\in S$ the intersection $X_s\cap \supp(Z)$ is a proper analytic subset of $X_s$, and 
$$\supp({\cal T}ors({\cal F}_s))\subset B({\cal F}_s)= X_s\cap \supp(Z).$$
 \item  The following two conditions are equivalent
 \begin{enumerate}
 \item 	$\supp({\cal T}ors({\cal F}_s))$ is either empty or 0-dimensional,
 \item $X_s\cap Z$ is either empty or 0-dimensional.
 \end{enumerate}
 \item The set
 $$S_0:=\{s\in S|\ {\cal F}_s\hbox{ is torsion-free or has 0-dimensional torsion}\}
 $$
 is Zariski open and dense in $S$.
  \end{enumerate}
\end{lm}
\begin{proof} (1) Since $\mathscr{F}$ is a rank 1 sheaf, flat over $S$, it follows that for any $s\in S$ the restriction ${\cal F}_s$ of $\mathscr{F}$ to $X_s$ is a rank 1 sheaf on $X_s$. To see this, it suffices to restrict local free resolutions of ${\cal F}$ to fibres, making use of its flatness over $S$.
\vspace{2mm} \\
(2) Since  $\dim(\mathscr{F}(s,x))=\dim({\cal I}_Z(s,x))\geq 2$ for any $(s,x)\in Z$  and ${\cal F}_s$ is rank 1 sheaf on $X_s$, it follows that $X_s\cap Z$ is a proper analytic subset of $X_s$, ${\cal F}_s$ is locally free of rank 1 on $X_s\setminus (X_s\cap Z)$, and  $B({\cal F}_s)= X_s\cap Z$.  Here we have used the natural isomorphism ${\cal F}(x,s)={\cal F}_s(x)$. This proves (2). 
\vspace{2mm} \\
(3) Restricting the short exact sequence
$$0\to \mathscr{F}\to \mathscr{F}^\we\to \mathscr{F}^\we_Z\to 0
$$
to $X_s$, we get an exact sequence
$$0\to \mathscr{F}^\we\otimes {\cal T}or_1^{{\cal O}_{S\times X}}({\cal O}_Z,{\cal O}_{X_s})\to {\cal F}_s\to (\mathscr{F}^\we)_s\to \mathscr{F}^\we_{Z\cap X_s}\to 0.
$$
Using (2) we see that the image of the monomorphism 
$$\mathscr{F}^\we\otimes {\cal T}or_1^{{\cal O}_{S\times X}}({\cal O}_Z,{\cal O}_{X_s})\to {\cal F}_s$$
  is precisely ${\cal T}ors({\cal F}_s)$. 
\vspace{-1mm} \\

Suppose that $(a)$ holds but $(b)$ does not, i.e. $X_s\cap  Z$ contains a 1-dimensional irreducible component $Y$. The sheaf morphism ${\cal F}_s\to \{\mathscr{F}^\we\}_s$ is a monomorphism  at any point  $(s,x)\in Y\setminus \supp({\cal T}ors({\cal F}_s))$. Using the short exact sequence
$$0\to \mathscr{F}=\mathscr{F}^\we\otimes{\cal I}_Z\to \mathscr{F}^\we\to \mathscr{F}^\we_Z\to 0
$$
and the flatness criterion \cite[p. 150]{Mat}, it follows  that $\mathscr{F}^\we_Z$ (or equivalently ${\cal O}_Z$) is flat over $S$ at any point  $(s,x)\in Y\setminus \supp({\cal T}ors({\cal F}_s))$.  This shows that $Z$ is flat over $S$ at any such point. Using the invariance of the fibre dimension in a flat morphism, this will contradict $\mathrm{codim}_{(s,x)}(Z)\geq 2$. The implication $(b)\Rightarrow(a)$ follows from claim (2) proved above. 
\vspace{2mm} \\ 
(4) By the fibre dimension semicontinuity theorem \cite[Theorem p. 137]{Fi} applied to the projection $Z\to S$, it follows that the set $Z_1$ of points   $(s,x)\in Z$ for which $\dim_{(s,x)}(X_s\cap Z)=1$ is Zariski closed in $Z$, hence its projection on $S$ will also be  Zariski closed.   On the other hand, by the dimension formula \cite[Corollary p. 138]{Fi} and the assumption on the codimension of $Z$, it follows that this projection has codimension $\geq 1$ at any point.
\end{proof}

 \begin{pr} \label{newsigma} Let $X$ be a complex surface, $S$ be a  reduced complex space, and $\mathscr{F}$ be a torsion-free rank 1 coherent sheaf on $S\times X$, flat over $S$. Denote by ${\cal F}_s$ the restriction of $\mathscr{F}$ to $X_s:=\{s\}\times X$.  Then

 \begin{enumerate}
 \item The sets 
 \begin{align*}
 S_{\mathrm{tf}}&:=\{s\in S|\ {\cal F}_s\hbox{ is torsion-free}\},\\
 S_0&:=\{s\in S|\ {\cal F}_s\hbox{ is torsion-free or has 0-dimensional torsion}\}
 \end{align*}
 are Zariski  open, and dense in $S$.
 \item For any $s\in S_0$ one has
 $$\Supp\big({\cal T}ors({\cal F}_s)\big)\leq \Supp \big( {{\cal F}_s^0}^\we/{\cal F}_s^0\big),
 $$
 where ${\cal F}_s^0:={\cal F}_s/{\cal T}ors({\cal F}_s)$.
 \item \label{cont} The map $\Sg_{\mathscr{F}}:S_0\to S^*(X)$ given by
 $$\Sg_{\mathscr{F}}(s):=\Supp \big( {{\cal F}_s^0}^\we/{\cal F}_s^0\big)-\Supp\big({\cal T}ors({\cal F}_s)\big)
 $$
 is continuous.
  \end{enumerate}

 \end{pr}
 
 \begin{proof} Replacing $\mathscr{F}$ by  its pull-back to $S'\times X$, where $S'\to S$ is a desingularisation of $S$, we may suppose that $S$ is smooth.  Note that the pull-back  $\mathscr{F}'$ of  $\mathscr{F}$  to $S'\times X$ will still be torsion-free, which follows using \cite[Lemma 5.6, p. 43]{Fu}. Since $\Sg_{\mathscr{F}'}$ is constant on the fibers of the proper map $S'\to S$, the continuity of $\Sg_{\mathscr{F}}$ will follow  from the continuity of $\Sg_{\mathscr{F}'}$. 
  \vspace{2mm} \\  
(1) The fact that $S_{\mathrm{tf}}$ is  Zariski  open and dense follows from \cite[Lemma 5.6, p. 43]{Fu}.  For $S_0$ the statement follows from Lemma \ref{equiv} above.
\\ \\
(2)  In order to prove the claim, let $s\in S_0$, and let $D\subset S_0$ be an embedded disk centered at $s$  such that $D\cap S_\tf\ne \emptyset$.  
Note that for such a disk $D$, the restriction $\mathscr{G}$ of $\mathscr{F}$ to $D\times X$ will also satisfy the assumptions of Proposition \ref{newsigma}. This follows   using \cite[Lemma 5.6, p. 43]{Fu}. 

Denote by $W\subset D\times X$ the complex subspace defined by the equality $\mathscr{G}=\mathscr{G}^\we\otimes {\cal I}_W$. Noting that $\Sg_{\mathscr{G}}=\Sg_{{\cal I}_W}$, it suffices to prove the claim for $\mathscr{G}={\cal I}_W$. Let $s\in D$. Tensoring    the short exact sequence
$$0\to {\cal I}_W\to {\cal O}_{D\times X}\to {\cal O}_W\to 0
$$
by  ${\cal O}_{X_s}$ gives
\begin{equation}\label{BES} 0\to {\cal T}or_1({\cal O}_W,{\cal O}_{X_s})\to {\cal I}_W\otimes{\cal O}_{X_s}\to {\cal O}_{X_s}\to {\cal O}_{W\cap X_s}\to 0, 
\end{equation}
hence (since we assumed $\mathscr{G}={\cal I}_W$) we get
$${\cal T}ors({\cal G}_s)={\cal T}or_1({\cal O}_W,{\cal O}_{X_s}).
$$
On the other hand, tensoring with ${\cal O}_W$ the short exact sequence
$$0\to {\cal I}_{X_s}\to {\cal O}_{D\times X}\to {\cal O}_{X_s}\to 0
$$
we get the exact sequence
$$0\to {\cal T}or_1({\cal O}_W,{\cal O}_{X_s})\to {\cal O}_W\otimes {\cal I}_{X_s}\to {\cal O}_W\to {\cal O}_{W\cap X_s}\to 0\ .
$$

Using Corollary 1 of \cite[p. 504]{Hi} we can find a proper analytic subset $M\subset D$ such that the closure $\Wg$ of $W\cap \big((D\setminus M)\times X\big)$ in $W$  is flat over $D$. The support of the kernel ${\cal U}:=\ker({\cal O}_W\to {\cal O}_{\Wg})$ is contained in the set $W\cap (M\times X)$, which is 0-dimensional by Lemma \ref{equiv} (3).  

Putting ${\cal T}:={\cal T}or_1({\cal O}_W,{\cal O}_{X_s})={\cal T}or({\cal G}_s)$,  taking into account the flatness of ${\cal O}_\Wg$ over $D$, and the fact that ${\cal I}_{X_s}$ is locally free, we get a commutative diagram
\begin{equation}\label{NiceDiag}
\begin{diagram}[s=6mm, midshaft]
&& 0 & &0 && 0 &&0  \\
&&\dTo&&\dTo &&\dTo && \dTo && \\
0 &\rTo & {\cal T}'	 &\rTo& {\cal U}\otimes {\cal I}_{X_s} &\rTo &{\cal U}& \rTo  &{\cal U}\otimes {\cal O}_{X_s}&\rTo & 0\\
&&\dTo&&\dTo &&\dTo && \dTo && \\
0 &\rTo  &{\cal T}	 &\rTo& {\cal O}_W\otimes {\cal I}_{X_s} &\rTo & {\cal O}_W &\rTo &{\cal O}_{W\cap X_s}&\rTo & 0\\
&&\dTo&&\dTo &&\dTo && \dTo && \\
0 &\rTo &0 &\rTo &{\cal O}_{\Wg}\otimes {\cal I}_{X_s} &\rTo & {\cal O}_\Wg &\rTo &{\cal O}_{\Wg\cap X_s}&\rTo & 0\\
&&&&\dTo &&\dTo && \dTo && \\
&&  & &0 && 0 &&0  \\
\end{diagram}
\end{equation}
with exact rows   and   exact columns. Here ${\cal T}'$ is defined as the kernel of the natural morphism ${\cal U}\otimes {\cal I}_{X_s} \to {\cal U}$; this shows in particular that the obvious inclusion ${\cal T}'\to {\cal T}$ is an isomorphism, hence the first  column is also exact, as claimed. We have the equalities 
\begin{equation}\label{W}
\begin{split}
\Supp \big( {{\cal G}_s^0}^\we/{\cal G}_s^0\big)=\Supp ({\cal O}_{W\cap X_s})&=\Supp({\cal U}\otimes {\cal O}_{X_s})+\Supp({\cal O}_{\Wg\cap X_s})\\
&=\Supp({\cal T}')+\Supp({\cal O}_{\Wg\cap X_s}),	
\end{split}
\end{equation}
% 
%
%$$\Supp \big( {{\cal G}_s^0}^\we/{\cal G}_s^0\big)-\Supp\big({\cal T}ors({\cal G}_s)\big)=\Supp ({\cal O}_{W\cap X_s})-\Supp({\cal T})=
%$$
%%
%$$
%=\Supp ({\cal O}_{W\cap X_s})-\Supp({\cal T}')=\Supp ({\cal O}_{W\cap X_s})-\Supp({\cal U}\otimes {\cal O}_{X_s})=\Supp({\cal O}_{\Wg\cap X_s}),
%$$
%
where the first equality follows from (\ref{BES}), the second from the last vertical exact sequence,   and for the third we used the formula $\Supp({\cal T}')=\Supp({\cal U}\otimes {\cal O}_{X_s})$ provided by the  first horizontal exact sequence in (\ref{NiceDiag}). Recalling ${\cal T}'\simeq {\cal T}={\cal T}ors({\cal G}_s)$, this proves (2).\\ 
\\
(3) In fact we will prove that
 \vspace{2mm}\\
 {\bf Claim:} {\it If $S$ is smooth, then   $\Sg_{\mathscr{F}}$ is  analytic.}
 \\ 
 
 We prove first the analyticity of the restriction of the map $\Sg_{\mathscr{F}}$ to a disk $D\subset S_0$ with $D\cap S_{\mathrm{tf}}\ne\emptyset$ as in the proof of (2). By formula (\ref{W}) we have  
$$\Sg_{\mathscr{F}}(s)=\Sg_{\mathscr{G}}(s)=\Supp \big( {{\cal G}_s^0}^\we/{\cal G}_s^0\big)-\Supp({\cal T}ors({\cal G}_s))=\Supp({\cal O}_{\Wg\cap X_s}),
$$
for any $s\in D$. Using the flatness of $\Wg$ over $D$, we get the analyticity of the restriction $\resto{\Sg_{\mathscr{F}}}{D}$.\\

Let now $Z\subset S\times  X$ be the complex subspace of $S\times X$ defined by $\mathscr{F}= \mathscr{F}^\we\otimes{\cal I}_Z$ (see Lemma  \ref{equiv}). Using the exact sequence 
$$0\to \mathscr{F}=\mathscr{F}^\we\otimes{\cal I}_Z\to \mathscr{F}^\we\to \mathscr{F}^\we_Z\to 0
$$
 and the flatness criterion \cite[p. 150]{Mat} as in the proof of Lemma \ref{equiv} (3), we see that $Z':=Z\cap (S_{\tf}\times X)$ is a (possibly empty) subspace of $S_\tf\times X$ which is finite and flat over $S_\tf$. Note that for any $s\in S_\tf$ one has ${\cal T}ors({\cal F}_s)=0$, and
 $$\Sg_{\mathscr{F}}(s)=\Supp({\cal F}_s^\we/{\cal F}_s)=\Supp({\cal O}_{Z_s}).
 $$
This proves that the restriction $\resto{\Sg_{\mathscr{F}}}{S_\tf}$ is analytic, because it coincides with the composition of the morphism $S_\tf\to {\cal D}ou(X)$ defined by $Z'$ with the canonical morphism ${\cal D}ou(X)\to S^*(X)$ from the Douady moduli space of 0-dimensional complex subspaces to the Barlet moduli space of 0-dimensional cycles of $X$. On the other hand the closure of $Z'$ in $S\times X$ defines a (possibly trivial) cycle $\Zg$ of $S\times X$ given by
$$\Zg=\sum_{\substack{C  \hbox{ \small irred. comp. of }Z_{\mathrm{red}}  \\ C\cap Z'_{\mathrm{red}}\ne\emptyset}} n_C\ C,
$$
where $n_C$ is the multiplicity of $C\cap Z'_{\mathrm{red}}$ in $Z'$.  Note that $\Zg$ is 0-equidimensional over $S_0$. Indeed, one has $\Zg\subset Z$, and $Z$ is already 0-equidimensional over $S_0$ by Lemma \ref{equiv} (3).  Therefore  $\Zg$ defines an analytic map $\Sg_{\Zg}:S_0\to S^*(X)$ extending $\resto{\Sg_{\mathscr{F}}}{S_\tf}$ \cite[Theorem 2]{Ba}. 

We claim that for any disk $D\subset \S_0$ with $D\cap S_{\mathrm{tf}}\ne\emptyset$ one has
$$\resto{\Sg_{\mathscr{F}}}{D}=\resto{\Sg_{\Zg}}{D}.$$
Indeed these restrictions are both analytic and coincide over $D\cap  S_{\mathrm{tf}}$. Since $S_0$ is covered by such disks, this shows that $\Sg_{\mathscr{F}}=\Sg_{\Zg}$, hence $\Sg_{\mathscr{F}}$ is analytic as claimed.
 \end{proof}
\begin{pr}\label{PropSS}
Let $(X,g)$ be a Gauduchon surface, ${\cal D}$ be a holomorphic line bundle on $X$, $c\in\Z$, $U$ be a reduced complex space, and $\mathscr{F}$ be a  sheaf   on $U\times X$, flat over $U$, such that 
\begin{itemize}
\item  $\det({\cal F}_u)\simeq {\cal D}$ for any $u\in U$.
\item For any $u\in U$ the sheaf ${\cal F}_u$ on $X$ defined by the restriction of $\resto{\mathscr{F}}{\{u\}\times X}$   is torsion-free with $c_2({\cal F}_u)=c$.
\end{itemize}
Then the set
$$U^{\sgs}:=\{u\in U|\ {\cal F}_u\hbox{ is properly slope semi-stable}\}.
$$ 
is closed in $U^\ss$.
\end{pr}
\begin{proof}
Let $(u_n)_n$ be a sequence in $U^\sgs$ converging  to a point $u_\infty\in U^\ss$.  Put ${\cal F}_n:={\cal F}_{u_n}$.  We have to show that ${\cal F}_\infty:={\cal F}_{u_\infty}$ is  properly semi-stable, too.	Let 
$$0\to {\cal A}_n\textmap{a_n}{\cal F}_n\textmap{b_n} {\cal B}_n\to 0$$
be   a destabilising exact sequence of ${\cal F}_n$, with ${\cal A}_n$, ${\cal B}_n$ torsion-free rank 1 coherent sheaves. Put ${\cal L}_n:={\cal A}_n^\we$, ${\cal M}_n:={\cal B}_n^\we$, and let $\Ag_n$, $\Bg_n$ be 0-dimensional complex subspaces of $X$ such that ${\cal A}_n={\cal L}_n\otimes{\cal I}_{\Ag_n}$, ${\cal B}_n={\cal M}_n\otimes {\cal I}_{\Bg_n}$.

One has
$$c=c_2({\cal F}_n)=c_1({\cal L}_n)c_1({\cal M}_n)+|\Ag_n|+|\Bg_n|=$$
$$= \frac{1}{4}\left[c_1({\cal D})^2-(c_1({\cal L}_n)-c_1({\cal M}_n))^2\right]+|\Ag_n|+|\Bg_n|.
$$
 Since $\deg_g({\cal L}_n)=\deg_g({\cal M}_n)$, it follows that the Einstein constant of a  Hermitian-Einstein metric on the line bundle ${\cal L}_n^\smvee\otimes{\cal M}_n$ vanishes, and therefore the Chern form of this metric is a primitive form of type (1,1).  Therefore the harmonic representative of the cohomology class $c_1({\cal L}_n)-c_1({\cal M}_n)$ is  ASD. Endowing $H^2(X,\R)$ with the $L^2$ inner product induced by Hodge  isomorphism, we get 
\begin{equation}\label{bound}\frac{1}{4}\|c_1^\R({\cal L}_n\otimes {\cal M}_n^\smvee)\|_{L^2}^2+|\Ag_n|+|\Bg_n|=c-\frac{1}{4}c_1({\cal D})^2\ \forall n\in\N.
\end{equation}

Formula (\ref{bound}) shows in particular that there exists a finite set $H\subset \NS(X)$ such that 	$c_1({\cal L}_n) -c_1({\cal M}_n)\in H$ for any $n\in\N$. Since $c_1({\cal L}_n) +c_1({\cal M}_n)=c_1({\cal D})$ is constant, it follows that $c_1({\cal L}_n)$, $c_1({\cal M}_n)$ both vary in finite sets hence, passing to a subsequence if necessary,  there exist classes $\lg$, $\mg\in\NS(X)$ such that 
$$c_1({\cal L}_n)=\lg,\ c_1({\cal M}_n)=\mg\ \forall n\in\N.
$$

Similarly, using (\ref{bound}) we see that (passing to a smaller subsequence if necessary) there exists $\ag$, $\bg\in\N$ such that
$$|\Ag_n|=\ag,\ |\Bg_n|=\bg,\ \forall n\in\N.
$$
Denote by ${\cal D}ou_k(X)$ the Douady space of 0-dimensional complex subspaces of length $k$, and by ${\cal I}_k$ the tautological ideal sheaf on ${\cal D}ou_k(X)\times X$. On the product 
$$\Pic^\lg(X)\times {\cal D}ou_\ag(X)\times U\times X$$ consider the sheaves 
$$p_{14}^*(\mathscr{L})\otimes p_{24}^*({\cal I}_\ag),\ p_{34}^*(\mathscr{F}),
$$ 
where $\mathscr{L}$ is a Poincaré line bundle on $\Pic(X)\times X$. These sheaves are flat over $\Pic^\lg(X)\times {\cal D}ou_\ag(X)\times U$. Using a standard semicontinuity theorem for the Ext functors \cite[Satz 3, p. 147]{BPS}, it follows that the Brill-Noether set
\begin{equation}\label{BN}
BN:=\{(y,\Ag,u)\in \Pic^\lg(X)\times {\cal D}ou_\ag(X)\times U|\ \Hom({\cal L}_y\otimes {\cal I}_{\Ag}, {\cal F}_u)\ne 0\}
\end{equation}
is a (closed) analytic set in $\Pic^\lg(X)\times {\cal D}ou_\ag(X)\times U$.
Recall that the degree map  $\deg_g:\Pic(X)\to\R$ is proper on any connected component of $\Pic(X)$ (see \cite[Corollary 1.3.12]{LT}). This implies that the intersection $BN\cap\big((\deg_g\circ p_1)^{-1}(\frac{1}{2}\deg_g({\cal D}))\big)$ is proper over $U$. Put $y_n:=[{\cal L}_n]$. Since 
$$(y_n,\Ag_n, u_n)\in BN\cap\big((\deg_g\circ p_1)^{-1}(\frac{1}{2}\deg_g({\cal D}))\big)\ \forall n\in\N,
$$
and $u_n\to u_\infty$ it follows that, passing to a smaller subsequence if necessary, the sequence of triples  $(y_n,\Ag_n, u_n)$ converges  to a triple $(y_\infty,\Ag_\infty, u_\infty)$ belonging to the same intersection, in particular
\begin{equation}\label{DestLim}
\Hom({\cal L}_{y_\infty}\otimes{\cal I}_{\Ag_\infty},{\cal F}_\infty)\ne 0.	
\end{equation}
Since $\deg_g({\cal L}_{y_\infty})=\frac{1}{2}\deg_g({\cal D})$, this shows that ${\cal L}_{y_\infty}\otimes{\cal I}_{\Ag_{y_\infty}}$ destabilises ${\cal F}_\infty$. Since ${\cal F}_\infty$ is semi-stable by assumption, if follows that   $u_\infty\in U^\sgs$.
\end{proof}

\begin{re}
By the same method  we get a pair $(z_\infty,\Bg_\infty)\in\Pic^\mg(X)\times {\cal D}ou_\bg(X)$ such that 
$$\Hom({\cal F}_\infty, {\cal L}_{z_\infty}\otimes{\cal I}_{\Bg_\infty})\ne 0.
$$
More precisely, taking into account that for any $n\in \N$ we have $b_n\circ a_n=0$,  we obtain non-trivial morphisms $a_\infty: {\cal L}_{y_\infty}\otimes{\cal I}_{\Ag_{\infty}}\to {\cal F}_\infty$, $b'_\infty: {\cal F}_\infty\to  {\cal L}_{z_\infty}\otimes{\cal I}_{\Bg_\infty}$ such that $b'_\infty\circ a_\infty=0$. Note that we cannot conclude that $\im(a_\infty)=\ker(b'_\infty)$ or that $b'_\infty$ is an epimorphism.
\end{re}

\begin{re} The proof of Proposition \ref{PropSS} yields a destabilising monomorphism 
$${\cal L}_{y_\infty}\otimes{\cal I}_{\Ag_\infty}\to {\cal F}_\infty,$$
but the corresponding quotient, say ${\cal Q}_\infty$, might have torsion. If the torsion subsheaf ${\cal T}_\infty$ of ${\cal Q}_\infty$ is non-trivial, then, by Corollary \ref{sigma} (2), its support must be  0-dimensional.

\end{re}
\begin{pr}\label{PropCont1}
Under the assumptions and with the notations of Theorem \ref{CT}, the map $\Phi_{\cal F}$ is continuous on $U^\sgs$.	
\end{pr}

\begin{proof}
We prove that  any convergent sequence $(u_n)_n$ in $U^\sgs$  has a subsequence $(u_{n_k})_k$ such that $\lim_{k\to\infty} \Phi_{\cal F}(u_{n_k})=\Phi_{\cal F}(\lim_{n\to \infty} u_n)$. Let $(u_n)_n$ be such a sequence and $u_\infty \in U^\sgs$ be its limit.  Define 
$${\cal F}_n,\ {\cal F}_\infty,\ a_n:{\cal A}_n\to {\cal F}_n,\ b_n:{\cal F}_n\to {\cal B}_n,\ \Ag_n,\ \Bg_n,\ {\cal L}_n,\ {\cal M}_n,\ y_n$$
 as in the proof of Proposition \ref{PropSS}.
Using the same method as in the proof of this proposition,  we can find
\begin{itemize}
\item classes $\lg$, $\mg\in\NS(X)$,
\item  non-negative integers  $\ag$, $\bg\in\N$, 	
\end{itemize}
such that (passing to a subsequence if necessary) the following holds:
$$c_1({\cal L}_n)=\lg,\ c_1({\cal M}_n)=\mg,\ |\Ag_n|=\ag,\ |\Bg_n|=\bg\ \forall n\in\N.
$$
Consider again the sheaves 
$$\mathscr{A}:=p_{14}^*(\mathscr{L})\otimes p_{24}^*({\cal I}_\ag),\ \mathscr{F}':=p_{34}^*(\mathscr{F}) $$
over $\Pic^\lg(X)\times {\cal D}ou_\ag(X)\times U\times X$, where $\mathscr{L}$ is a Poincaré line bundle on $\Pic(X)\times X$. Since the sheaf $p_{34}^*(\mathscr{F})$ is flat over $S:=\Pic^\lg(X)\times {\cal D}ou_\ag(X)\times U$, it follows that the contravariant functor
$$\underline{\Hom}_S\big(\mathscr{A},  \mathscr{F}'\big) :{\cal A}n_S\to {\cal S}ets
$$
defined by 
$$(\tau:T\to S)\mapsto \Hom_{(T\times X)} \big((\tau\times\id_X)^*(\mathscr{A}),(\tau\times\id_X)^*( \mathscr{F}')\big).$$
 is representable by a linear fibre space over $S$, say $\chi:H\to S$ (see \cite[sec. 3.2]{Fl}).   
 
 Using the object $\big(\{s\}\hookrightarrow S\big)\in {\cal O}b({\cal A}n_S)$ we get for any $s\in S$ an identification $H_s=\Hom({\cal A}_s,{\cal F}'_s)$. Moreover, let $\theta:T\to H$ be a morphism, and   $\tau:=\chi\circ\theta$. Regarding $\theta$ as a morphism $\tau\to \chi$ in the category ${\cal A}n_S$ we obtain an associated   morphism 
 $$\tilde \theta\in \Hom_{(T\times X)} \big((\tau\times\id_X)^*(\mathscr{A}),(\tau\times\id_X)^*( \mathscr{F}')\big),$$
and using the commutative diagram
 $$
 \begin{diagram}[s=7mm, midshaft]
 \{t\}&\rInto &T \\
 ^{\simeq}\dTo & \rdTo^{\tau_t} & \dTo_{\tau} \\
 \{\tau(t)\}&\rInto &S
 \end{diagram}
 $$
in the category ${\cal A}n_S$, we see that for any $t\in T$ the restriction of $\tilde\theta$ to $\{t\}\times X$ coincides (via the identification $\{t\}\times X=\{\tau(t)\}\times X$) with  $\theta(t)\in H_{\tau(t)}$.\\

  Note that the support of $\chi:H\to S$ is precisely the Brill-Noether locus defined by   formula (\ref{BN}). Consider the projectivisation $h:\P(H)\to S$ of $\chi$.  The non-trivial morphism $a_n$ defines a point $[a_n]\in \P(H)$  lying over $(y_n,\Ag_n,u_n)\in S$. Using a properness argument, similar to the one used in Proposition \ref{PropSS}, we see that (passing to a subsequence if necessary) $([a_n])_n$ converges in $\P(H)$ to a point $\alpha_\infty$ lying over a triple $(y_\infty,\Ag_\infty,u_\infty)\in S$, where $\deg_g({\cal L}_{y_\infty})=\frac{1}{2}\deg_g({\cal D})$.  
 
 Let $\Sigma\subset \P(H)$ be a sufficiently small open neighborhood of $\alpha_\infty$ in $\P(H)$ such that a section $\nu:\Sigma\to H\setminus\hbox{(0-section)}$ of the canonical projection $\pi:H\setminus\hbox{(0-section)}\to\P(H)$ exists (see Lemma \ref{section} below). Regarding $\nu$ as a nowhere vanishing $H$-valued map over $S$, and using the fact that the linear space $H\to S$ represents the functor $\underline{\Hom}_S\big(\mathscr{A},  \mathscr{F}'\big)$,   we obtain a morphism 
$$\tilde\nu\in \Hom_{\Sigma\times X} \big((h_\Sigma\times\id_X)^*(\mathscr{A}),(h_\Sigma\times\id_X)^*(\mathscr{F}')\big),
$$    
where $h_\Sigma:=\resto{h}{\Sigma}$. For any $y\in\Sigma$ the restriction of $\tilde\nu$ to a fibre $\{y\}\times X$ coincides with $\nu(y)$, hence it is non-trivial. 

Put    $\tilde{\mathscr{A}}:=(h_\Sigma\times\id_X)^*(\mathscr{A})$,  $\tilde{\mathscr{F}}:=(h_\Sigma\times\id_X)^*(\mathscr{F}')$, and let 
$$0\to \tilde{\mathscr{A}}\textmap{\tilde\nu} \tilde{\mathscr{F}}\to \tilde{\mathscr{Q}}\to 0
$$
be the corresponding short exact sequence on  $\Sigma\times X$. Since   $\tilde{\mathscr{F}}$ is flat over $\Sigma$, and $\tilde\nu$ is fibrewise injective, it follows \cite[p. 150]{Mat} that $\tilde{\mathscr{Q}}$ is also flat over $\Sigma$. The restriction of this short exact sequence to $\{a_\infty\}\times X$ is
\begin{equation}\label{ESinfty}
0\to {\cal L}_{y_\infty}\otimes{\cal I}_{\Ag_\infty}\textmap{a_\infty} {\cal F}_\infty\textmap{b_\infty} {\cal Q}_\infty\to  0,	
\end{equation}
where $a_\infty$ is a representative of $\alpha_\infty$ (in other words $\alpha_\infty=[a_\infty]$), and, via   suitable identifications 
$${\cal L}_n={\cal L}_{y_n},\ {\cal Q}_n:=\qmod{{\cal F}_n}{a_n\big({\cal L}_{y_n}\otimes{\cal I}_{\Ag_n}\big)}={\cal B}_n,$$
its restriction to $\{[a_n]\}\times X$ is
$$0\to {\cal L}_{y_n}\otimes{\cal I}_{\Ag_n}\textmap{a_n} {\cal F}_n\textmap{b_n} {\cal Q}_n\to 0.
$$
 Therefore the quotient ${\cal Q}_\infty$ of ${\cal F}_\infty$ by ${\cal L}_{y_\infty}\times {\cal I}_{\Ag_\infty}$ is the ``limit" of the  sequence $({\cal B}_n)_n$ in a flat family of sheaves parameterized by $\Sigma$.

 Since $\lim_{n\to\infty} y_n=y_\infty$ the exact sequence (\ref{ESinfty}) gives
 $$P({\cal F}_\infty)=\big[{\cal L}_{y_\infty}\oplus({\cal D}\otimes{\cal L}_{y_\infty}^\smvee)\big]=\lim_{n\to \infty}\big[{\cal L}_{y_n}\oplus({\cal D}\otimes{\cal L}_{y_n}^\smvee)\big]=\lim_{n\to \infty}P({\cal F}_n).
 $$ 
On the other hand using  Corollary \ref{sigma} and the    exact sequence (\ref{ESinfty}), we obtain:
$$\sigma({\cal F}_\infty)=\Supp({\cal O}_{\Ag_\infty})+\Supp\big(\{{\cal Q}_{\infty}\}_0^\we/\{{\cal Q}_{\infty}\}_0\big)-
\Supp\big({\cal T}ors({\cal Q}_{\infty})\big).
$$
Taking into account that
$$\lim_{n\to\infty}\Ag_n= \Ag_\infty
$$
and that, by Proposition \ref{newsigma} (\ref{cont}), we have
\begin{align*}
\Supp\big(\{{\cal Q}_{\infty}\}_0^\we/\{{\cal Q}_{\infty}\}_0\big)&-
\Supp\big({\cal T}ors({\cal Q}_{\infty})\big)=\\
&=\lim_{n\to\infty} \Supp\big(\{{\cal Q}_{n}\}_0^\we/\{{\cal Q}_{n}\}_0\big)-
\Supp\big({\cal T}ors({\cal Q}_{n})\big)=\\ &=\lim_{n\to\infty} \Supp\big({\cal Q}_{n}^\we/{\cal Q}_{n}\big)=\lim_{n\to\infty} \Supp({\cal O}_{\Bg_n}),
\end{align*}
we obtain
$$\sigma({\cal F}_\infty)=\lim_{n\to\infty}\Supp({\cal O}_{\Ag_n}) + \Supp({\cal O}_{\Bg_n})=\lim_{n\to\infty} \sigma({\cal F}_n).
$$
\end{proof}
\begin{lm} \label{section} Let $S$ be a complex space, ${\cal H}$ be a coherent sheaf on $S$,  $H\to S$ the associated linear space, $H\setminus\hbox{(0-section)}$ the complement of the 0-section in $H$,  $\P(H)\to S$  be its projectivisation, and $\chi:H\setminus\hbox{(0-section)}\to\P(H)$ the canonical projection. Any point $[\zeta_0]\in \P(H)$ has an open neighborhood on which a section of $\chi$ exists. 	
\end{lm}
\begin{proof} Recall that one has a canonical identification $H_{s_0}={\cal H}(s_0)^\smvee$. 
Let $S_0$ be a sufficiently small neighbourhood of $s_0$ in  $S$ and $\eta\in {\cal H}(S_0)$ be a  section   such that $\langle \zeta_0, \eta(s_0)\rangle =1$.  The condition $\langle \zeta,\eta\rangle =1$ defines a locally closed subspace $H_\eta\subset H_{S_0}^\smvee$, and the restriction $\resto{\chi}{H_\eta}:H_\eta\to \P(H)_{S_0}$ identifies $H_\eta$ biholomorphically with an open neighborhood $\Sigma_\eta$ of $[\zeta_0]$ in $\P(H)$. The inverse of this  map defines a section of $\chi$ defined on  $\Sigma_\eta$.
 \end{proof}

 \section{Sequences of stable sheaves}\label{STCase}
 
 \subsection{The difficulty}\label{Pr2sect}
 
 The goal of this section is to prove the second statement needed in the proof of Theorem \ref{CT}:
 
 \begin{pr}\label{PropCont2}
 With the hypotheses  and notations of Theorem \ref{CT}, let $u\in U^\ss$, and $(u_n)_n$ be a sequence of 	$U^\ss$ converging to $u$ such that ${\cal F}_{u_n}$ is slope stable for any $n\in\N$. Then $(u_n)_n$ has a subsequence $(u_{n_k})_k$ such that $\lim_{k\to\infty} \Phi_{\mathscr{F}}(u_{n_k})= \Phi_{\mathscr{F}}(u)$.
 \end{pr}

This result will be obtained as follows.  Applying the Kobayashi-Hitchin correspondence to the stable bundles ${\cal F}_{u_n}^\we$ one obtains projectively ASD unitary connections $A_n$ on the underlying differentiable bundles $F_{u_n}$. We may suppose that the isomorphism type of these bundles is constant.  Using the Donaldson-Uhlenbeck compactness theorem  we get a unitary bundle $(V,K)$ with a projectively ASD  unitary connection $A$ such that, after passing to a subsequence, the sequence of gauge classes $([A_n])_n$ converges to $[A]$ on the complement of a finite set.  The main difficulty  is dealt with in   Theorem \ref{mainTh} in this article, which replaces \cite[Theorem 4.2.3]{Morg}, and  is the construction of a {\it non-trivial}  holomorphic morphism $g_\infty: {\cal F}_u^\we\to {\cal V}$, where ${\cal V}$ is the holomorphic bundle $(V,\bar\partial_A)$. This morphism allows one to ``compare" the bundle ${\cal F}_u^\we$ with the polystable bundle ${\cal V}$, obtained as a ``limit" of the sequence of neighbouring bundles  ${\cal F}_{u_n}^\we$. 
In this section we  focus on this difficulty, leaving the rest of the proof of Proposition \ref{PropCont2} to follow   the arguments in \cite[sections 4.2.2, 4.2.3] {Morg}. 

The morphism $g_\infty$ is obtained in a 2-step procedure: 
\begin{enumerate}[S1.]
\item  Normalizing  a sequence of bundle morphisms $g_n$ with respect to  the $L^4$-norm of the    complement $R'$ of the union of sufficiently small   balls   centered at the singularities of ${\cal F}_u$.	
\item Bootstrapping and elliptic estimates to prove the convergence of the obtained sequence $(\tilde g_n)_n$  of $L^4$-normalized morphisms.
\end{enumerate}

The input estimate in the second step is an $L^\infty$-estimate, which is obtained by combining the $L^4$-bound   with   an elliptic inequality for the pointwise squared norm $|\tilde g_n|^2$, and using a standard result  in the theory of elliptic operators. Unfortunately the obtained  $L^\infty$-estimate only holds on a smaller set $R\Subset R'$, so one obtains convergence of the sequence  $(\tilde g_n)_n$ only on $R$ (more precisely on the complement in $R$ of a union a balls around finitely many bubbling points). A priori one has no control over the pointwise norm $|\tilde g_n|$ on    $R'\setminus R$, so the $L^4(R')$-norm of  $\tilde g_n$ might become concentrated   in   this set, leaving $\tilde g_n$   to  converge to zero in $R$. In other words, passing from $R$ to $R'$ in the second step, one will lose the uniform bound from below (by a positive number) given by the first step, and the non-vanishing of the limit will no longer be guaranteed.  This is the second gap in \cite{Morg} we mentioned in the introduction.  We overcome this difficulty by  noting that $R'\setminus R$ is a union of annuli, and make use of our a priori estimates for holomorphic sections in bundles on an annulus, which are proved in  section \ref{AnnEst} of the appendix.

\subsection{The proof}  
 For any $v\in U$ we have an exact sequence
$$0\to {\cal F}_v\to {\cal F}^{\we}_v\to {\cal Q}_v\to 0,
$$
where ${\cal F}^{\we}_v$ stands for the double dual of ${\cal F}_v$. We denote by $S_v\subset X$  the support of ${\cal Q}_v$, and by $F_v$ the underlying differentiable bundle of the locally free sheaf    ${\cal F}^{\we}_v$. Note that the second Chern class of $F_v$ might change when $v$ varies in $U$.

\begin{re}
For a point $(v,x)\in U\times X$ the following are equivalent:
\begin{enumerate}
\item 	$\mathscr{F}$ is singular at $(v,x)$.
\item ${\cal F}_v$ is singular at $x$.
\end{enumerate}
\end{re}
\begin{proof}
$\mathscr{F}$ is singular $(v,x)$ if and only if $\dim(\mathscr{F}(v,x))>\rk(\mathscr{F})$. But
$$\mathscr{F}(v,x)=\mathscr{F}_{(v,x)}\otimes_{{\cal O}_{U\times X,(v,x)}} \C_x=\big(\mathscr{F}_{(v,x)}\otimes_{{\cal O}_{U\times X,(v,x)}} {\cal O}_{X_v,x} \big)\otimes_{{\cal O}_{X_v,x}} \C_x=$$
$$=({\cal F}_v)_x\otimes_{{\cal O}_{X_v,x}} \C_x={\cal F}_v(x).
$$

It suffices to recall that, since $\mathscr{F}$ is flat over $U$, one has $\rk({\cal F}_v)=\rk(\mathscr{F})$ for any $v\in U$. This can be seen using local free resolutions of $\mathscr{F}$.
\end{proof}

Therefore, putting $\mathscr{S}:=\mathrm{Sing}(\mathscr{F})$, one has for any $v\in U$
\begin{equation}\label{SSu}
\mathscr{S}\cap X_v=\mathrm{Sing}({\cal F}_v)= S_v=\supp\big(\qmod{{\cal F}_v^{\we}}{{\cal F}_v}\big).
\end{equation}

\vspace{2mm}

Let $u\in U^\ss$, $(u_n)_n$ be a sequence of $U^\ss$  such that $\lim_{n\to\infty} u_n=u$.
Since the problem is local around $u$, we may suppose that $\det (\mathscr{F})\simeq p_X^*({\cal D})$, and we fix an isomorphism 
$$p_X^*({\cal D})\textmap{\simeq} \det (\mathscr{F}).$$
This isomorphism will induce isomorphisms ${\cal D}\textmap{\simeq}\det({\cal F}_u)$ for every $u\in U$.
 Fix a Hermitian-Einstein metric $\chi$ on ${\cal D}$, and denote by the same symbol its pull-back on $\det (\mathscr{F})$ and also its restrictions to $\det({\cal F}_u)$ for any $u\in U$. We will denote by $a$ the Chern connection of the pair $({\cal D},\chi)$. The underlying differentiable line bundle of ${\cal D}$ will be denoted by $D$. For $\varepsilon>0$ put
$$B_\varepsilon(S_u):=\bigcup_{x\in S_u} B_\varepsilon(x),
$$
where, for any $x\in S_u$,  $B_\varepsilon(x)$ denotes the inverse image of the ball $B_\varepsilon(0)\subset \C^2$ via a fixed chart $\cg_x:W_x\to B_1(0)$ centered at $x$. 

 We  choose $\varepsilon$ sufficiently small  such that for $p\ne q$ one gets disjoint open balls  $B_\varepsilon(p)$, $B_\varepsilon(q)$. Put
$$R_\varepsilon:= X\setminus B_\varepsilon(S_u).
$$

The projection $\pg:\mathscr{S}\to U$ is proper, and $\pg^{-1}(u)=\mathrm{Sing}({\cal F}_u)=S_u$. For any $\varepsilon>0$  the intersection $\mathscr{S}\cap(U\times B_\varepsilon(S_u))$ is an open neighbourhood  of  $\pg^{-1}(u)$, so, since $\pg$ is proper, there exists $V_\varepsilon$, an open neighbourhood of $u$ in $U$, such that $\pg^{-1}(V_\varepsilon)\subset U\times B_\varepsilon(S_u)$.  Equivalently,
$$\mathscr{S}\cap (V_\varepsilon\times X)\subset U\times B_\varepsilon(S_u);
$$
that is,
$$(V_\varepsilon \times R_\varepsilon)\cap\mathscr{S}=\emptyset.
$$
Therefore
\begin{re}
There exists $\eg>0$ such that for any  $\varepsilon\in (0,\eg]$, there exists an open neighbourhood $V_\varepsilon$ of $u$ in $U$ such that the restriction $\resto{\mathscr{F}}{V_\varepsilon\times R_\varepsilon}$ is locally free.	
\end{re}

By Remark \ref{SmoothBase} one can assume that $U$ is smooth.  Suppose $V_{\eg}$ is biholomorphic to a ball, and fix a biholomorphism $\bg:V_{\eg}\to B_1(0)\subset\C^{\dim(U)}$ mapping $u$ to 0. We may suppose that, for a monotone increasing function $\nu:(0,\eg]\to (0,1]$, one has $V_\varepsilon=\bg^{-1}(B_{\nu_\varepsilon}(0))$ for any $\varepsilon\in (0,\eg]$. Consider   a Hermitian metric on the bundle 
$$\resto{\mathscr{F}}{(U\times X)\setminus \mathscr{S}}$$
which induces $\chi$ on $\det(\mathscr{F})$. The determinant of the associated Chern connection coincides with the restriction of $p_X^*(a)$  to $(U\times X)\setminus \mathscr{S}$, so it is trivial in the horizontal directions.

For $(v,x)\in V_\varepsilon\times R_\varepsilon$ use parallel transport (with respect to this Chern connection) along the segments $\bg^{-1}\{(t\bg(v),x)| t\in[0,1]\}$ to identify   the restrictions $\resto{F_v}{R_\varepsilon}$, $\resto{F_u}{R_\varepsilon}$ of the underlying differentiable bundle $F$  of $\resto{\mathscr{F}}{(U\times X)\setminus \mathscr{S}}$ to $\{v\}\times R_\varepsilon$, $\{u\}\times R_\varepsilon$. One gets smooth  bundle isomorphisms (which induce the identity on the determinant line bundles)
$$\varphi_{\varepsilon,v}:\resto{F_u}{R_\varepsilon}\to \resto{F_v}{R_\varepsilon}\ \forall\varepsilon\in(0,\eg] \ \forall v\in V_\varepsilon
$$
such that, for $\varepsilon'<\varepsilon\leq \eg$, one has
$$\resto{\varphi_{\varepsilon',v}}{R_\varepsilon}=\varphi_{\varepsilon,v}\ \forall v\in V_{\varepsilon'}.
$$
Note that
\begin{re}\label{Conv}
Let $r\in(0,\eg]$. For any $\varepsilon\in(0,r]$ one has
$\lim_{v\to u} \resto{\varphi_{\varepsilon,v}}{R_r}=\id_{(\resto{F_u}{R_r})}
$ in the ${\cal C}^\infty( R_r)$-topology, all the maps here being regarded as maps $\resto{F_u}{R_r}\to \resto{F}{V_r\times R_r}$. 
\end{re}
%
%
%We refer to \cite[Lemma 4.2.2]{Morg} for the following topological result:
%\begin{lm}
%Let $x\in S_u$.  For any $\varepsilon\in(0,\eg]$ and any $v\in V_\varepsilon$ one has 	
%%
%$$\deg(\resto{\varphi_{\varepsilon,v}}{\partial \bar   B_\varepsilon(x)})=|{\cal Q}_{u,x}|-|\resto{{\cal Q}_{v}}{B_\varepsilon(x)}|.$$
%%.  
%\end{lm}
%%  
%In this lemma $\deg(\resto{\varphi_{\varepsilon,v}}{\partial \bar B})$ is defined using trivialisations of $F_v$  and $F_u$ on $\bar   B_\varepsilon(x)$.\\
%
%

There exists a decreasing sequence $(\varepsilon_n)_n\to 0$ such  that $u_n\in V_{\varepsilon_n}$. Put
$$\psi_n:=\varphi_{\varepsilon_n,u_n}: \resto{F_u}{R_{\varepsilon_n}}\to  \resto{F_{u_n}}{R_{\varepsilon_n}}
$$ 
For $v\in U$ denote by $\delta_v$ the integrable semi-connection on the bundle $F_v$ induced by the holomorphic structure ${\cal F}_v^{\we}$ (which is locally free on $X\setminus S_v$).
For any $n\in\N$ consider the integrable semi-connection
$$\delta^{u_n}_u:=\psi_n^*(\delta_{u_n})
$$
on the restriction $\resto{F_u}{R_{\varepsilon_n}}$. For    $r\in(0,\eg]$ let $n_r\in\N$ be such that $\varepsilon_n< r$ for any $n\geq n_r$. Therefore, for any $n\geq n_r$, the isomorphism $\psi_n$  and the semi-connection $\delta^{u_n}_u$ are defined on $R_r$.

Using Remark \ref{Conv} we see that for any fixed  $r\in(0,\eg]$ one has
\begin{equation}\label{delta-conv}
\lim_{\substack{ n\to\infty}}\delta^{u_n}_u=\delta_u\hbox{ in  ${\cal C}^\infty(  R_r)$} ,\\
	\end{equation}
the semi-connection on the left being defined on $R_r$ for any $n\geq n_r$.
Fixing a Hermitian metric $h$ on $F_u$ with $\det(h)=\chi$, we  have for any $r\in (0,\eg]$:
\begin{equation}\lim_{\substack{n\to\infty}}A(h,\delta^{u_n}_u)=A(h,\delta_u) \hbox{ in  ${\cal C}^\infty(  R_r)$}, 
	\end{equation}
the connection on the left being defined on $R_r$ for any $n\geq n_r$.
Putting $C:=\sup_X |F_{A(h,\delta_u)}|+1$ we see that
\begin{re}\label{boundRe}
For any  $r\in(0,\eg]$  there exists a natural number  $m_r\geq n_r$  such that 
\begin{equation}\label{BoundOnR}
\sup_{R_r}	|  F_{A(h,\delta^{u_n}_u)}|\leq C\ \forall n\geq m_r.  
\end{equation}

\end{re}

For every $n\in\N$, denote by $H_{u_n}$ a    Hermite-Einstein  metric on the holomorphic bundle ${\cal F}_{u_n}^\we$ which induces $\chi$ on  $\det({\cal F}_{u_n})$. This implies that  the  corresponding Chern connections $A_n:=A(H_{u_n},\delta_{u_n})$ are projectively ASD (i.e. $\big(F_{A_n}^0\big)^+=0$), and $\det(A_n)=a$ (via the fixed identification).  Passing to a subsequence if necessary we may suppose that the isomorphism type of $(F_{u_n},H_n)$ is independent of $n$; let $(F_0,H_0)$ be a	Hermitian 2-bundle on $X$ which is isomorphic to   $(F_{u_n},H_n)$ for any $n\in\N$, and fix a unitary isomorphism $\det(F_0,H_0)\textmap{\simeq}(D,\chi)$.

Using the Donaldson-Uhlenbeck compactness theorem \cite{DK} it follows that there exist 
\begin{enumerate}[\it i)]
\item unitary isomorphisms $\alpha_n:(F_0,H_0)\to (F_{u_n},H_{u_n})$ with $\det(\rho_n)=1$,
\item A Hermitian 2-bundle $(V,K)$ with an identification $\det(V,K)\textmap{\simeq}(D,\chi)$ on $X$,
\item A projectively ASD Hermitian connection $A$ on $V$ with $\det(A)=a$,
\item A finite set $\Sigma_0\subset X$, and a unitary isomorphism $\nu_0:\resto{F_0}{X\setminus \Sigma_0}\to \resto{V}{X\setminus\Sigma_0}$ with $\det(\nu_0)=1$,
\end{enumerate}
such that, passing to a subsequence if necessary, the sequence $(\alpha_n^*(A_n))_n$ converges to $\nu_0^*(A)$ in the Fréchet ${\cal C}^\infty$-topology of $X\setminus \Sigma_0$.

We take $\Sigma_0$ to be minimal with this property, so for any point $x\in \Sigma_0$ and for any ball  $B$ centered at $x$ one has 
\begin{equation}\label{OnTheBalls}
\liminf_{n\to\infty}\int_{\bar B} |F_{A_n}^0|^2 \vol_g\geq 8\pi^2.	
\end{equation}
We denote by ${\cal V}$ the holomorphic bundle associated with the pair $(V,\bar\partial_A)$. 

Note now that the obvious line bundle isomorphism $\det(F_u,h)\textmap{\simeq}\det(F_0,H_0)$ always has a unitary isomorphic lift
$$\rho:\resto{(F_u,h)}{X\setminus S_u}\to \resto{(F_0,H_0)}{X\setminus S_u}.
$$
 Indeed, when $S_u=\emptyset$, it follows by (\ref{SSu}) that $S_v=\emptyset$ and $F_u\simeq F_v$ for any $v$ in a neighbourhood of $u$. In this case $X\setminus S_u=X$, and the obstruction to the existence  of a lift $\rho$ as above vanishes.  When $S_u\ne \emptyset$ this obstruction obviously vanishes because $H^4(X\setminus S_u,\Z)=0$.  Now put
$$\rho_n:=\alpha_n\circ \rho:\resto{F_u}{X\setminus S_u}\to \resto{F_{u_n}}{X\setminus S_u},\  \nu:=\nu_0\circ \rho: \resto{F_u}{X\setminus (S_u\cup \Sigma_0)}\to \resto{V}{X\setminus (S_u\cup \Sigma_0)}.$$
With these notations we have
\begin{equation}\label{rho-prop}
\begin{gathered}
\rho_n^*(H_{u_n})=h,\ \det(\rho_n)\equiv 1 \hbox{  on } X\setminus S_u\\	
\lim_{n\to \infty}\resto{\rho_n^*(A(H_{u_n},\delta_{u_n}))}{X\setminus (S_u\cup \Sigma_0)}=\resto{\nu^*(A)}{X\setminus (S_u\cup \Sigma_0)} \hbox{ in }{\cal C}^\infty(X\setminus (S_u\cup \Sigma_0))
\end{gathered}
\end{equation}
It is convenient to define $\Sigma:=\Sigma_0\setminus S_u$, hence to write the union $S_u\cup \Sigma_0$ as a disjoint union $S_u\cup \Sigma$.
\vspace{2mm}

Put  
$$\delta_u'^{u_n}:=\rho_n^*(\delta_{u_n}).  
$$
With this notation one has:
\begin{equation}\label{delta'-conv}
\lim_{n\to\infty}\delta_u'^{u_n}=\nu^*(\bar\partial_A) \hbox{ in }{\cal C}^\infty(X\setminus (S_u\cup \Sigma)),
\end{equation}
\begin{equation}\label{delta'}
\rho_{n*}(A(h,\delta_u'^{u_n}))=A(H_{u_n},\delta_{u_n}).
\end{equation}
Put
\begin{equation}\label{DefJnGn}
\begin{gathered}
j_n=\rho_n^{-1} \psi_n:\resto{F_u}{R_{\varepsilon_n}}\textmap{\simeq} \resto{F_u}{R_{\varepsilon_n}},\\	
g_n=\nu\circ j_n:\resto{F_u}{R_{\varepsilon_n}\setminus \Sigma}\textmap{\simeq} \resto{V}{R_{\varepsilon_n}\setminus \Sigma}.
\end{gathered}	
\end{equation}
Putting $H^{u_n}_u:=\psi_n^*(H_{u_n})$, one has
\begin{equation}\label{jnPBs} j_{n*}(H^{u_n}_u) =\resto{h}{R_{\varepsilon_n}},\ j_{n*}(\delta_u^{u_n}) =\resto{\delta_u'^{u_n}}{R_{\varepsilon_n}}.
\end{equation}
This gives the equality of $K$-unitary connections
\begin{equation}\label{gnConn}
g_{n*}(A(H^{u_n}_u,\delta_u^{u_n}))=A(K,\nu_*(\delta_u'^{u_n}))
\end{equation}
on $\resto{V}{R_{\varepsilon_n}\setminus \Sigma}$. Taking into account the convergence property (\ref{delta'-conv}) we get:

\begin{equation}\label{ConvConn}
\lim_{n\to\infty} g_{n*}(A(H^{u_n}_u,\delta_u^{u_n}))=A.
\end{equation}
Let    $j_n^*$ be the adjoint of $j_n$ with respect to $h$, and put
$$\tau_n:=\tr (j_n^*\circ j_n)\in {\cal C}^\infty(R_{\varepsilon_n},\R_{>0}).$$ 
\begin{re}\label{tau-estimate} For any $r\in(0,\eg]$  and $n\geq m_r$ the function $\tau_n$ is defined on $R_r$, and one has
$$-i\Lambda\partial\bar\partial \tau_n\leq C \tau_n \hbox{ on } R_r,
$$
for some positive constant $C$.	
\end{re}
\begin{proof}
This follows from Corollary \ref{TrEstimateCo}	and Remark \ref{BoundOnR}, noting that, by (\ref{jnPBs}),     
$$j_n: \resto{F_u}{R_{\varepsilon_n}}\to \resto{F_u}{R_{\varepsilon_n}}$$
is a holomorphic morphism with respect to the pair $ (\delta_u^{u_n},\delta_u'^{u_n})$. 
\end{proof}

Using now   \cite[Theorem 9.20]{GT} we get:

\begin{lm}\label{Crr}
For any pair  $(r,r')\in(0,\eg]\times (0,\eg]$ with $r'<r$ there exists a constant $C_{r,r'}>0$   such that, for any $n\geq m_{r'}$, one has 
$$\sup_{R_r} \tau_n\leq C_{r,r'} \|\tau_n\|_{L^2(R_{r'})}.
$$	
\end{lm}
We can prove now:
\begin{thry}\label{mainTh} Fix an arbitrary Hermitian metric $h$ on $F_u$. 
With the notations and under the assumptions above the following holds after passing to a subsequence: 	 There exists a sequence $(\lambda_n)_n$ of positive constants, and a non-zero holomorphic map $g_\infty:{\cal F}_u^{\we}\to {\cal V}$ such that
\begin{enumerate}
\item 	For any compact 4-dimensional submanifold with boundary
$$\Gamma\subset X\setminus (S_u\cup \Sigma)$$
 one has 
$$\resto{(\lambda_n g_n)_*(A(H^{u_n}_u,\delta_u^{u_n}))}{\Gamma}\to \resto{A}{\Gamma}\hbox{ and } \resto{\lambda_n g_n}{\Gamma}\to \resto{g_\infty}{\Gamma}\hbox{ in } {\cal C}^\infty(\Gamma),$$
\item For every $q\in \Sigma$, and for any ball $B$ with center $q$ it holds
$$\int_{\bar B} | F_{A(H^{u_n}_u,\delta_u^{u_n})}^0|^2\vol_g  \geq 8\pi^2 \hbox{ for any sufficiently large }n\in \N.
$$ 
\end{enumerate}

\end{thry}

\begin{proof}
We take $r<\eg$   sufficiently small such that:\\
\begin{enumerate}[(i)]
\item  for any $x\in S_u$ one has $\bar B_r(x)\cap   \Sigma=\emptyset$,\\
\item for any $(x,y)\in S_u\times S_u$ with $x\ne y$ one has $\bar B_r(x)\cap \bar B_r(y)=\emptyset$,\\
\item ${\cal F}_u^{\we}$ and ${\cal V}$ are trivial on  $\bar B_r(x)$ for any $x\in S_u$. 
\end{enumerate}
\vspace{4mm}

For any $n\geq n_{ r/{2}}$ put 
$$\lambda_n:=\big\{\|\tau_n\|_{L^2(R_{ \frac{r}{2}})}\big\}^{-\frac{1}{2}}
$$
$$\tilde\rho_n=\lambda_n^{-1}\rho_n,\ \tilde j_n=\tilde\rho_n^{-1} \psi_n=\lambda_n j_n,\ \tilde g_n:=\nu\circ \tilde j_n=\lambda_n g_n,\ \tilde\tau_n=\tr (\tilde j_n^*\circ \tilde j_n)=\lambda_n^2\tau_n,
$$
and note that one has 
\begin{equation}\label{eq1}
\|\tilde j_n\|_{{L^4(R_{ \frac{r}{2}})}}=\| \tilde \tau_n\|_{L^2(R_{\frac{r}{2}})}=1.	
\end{equation}
Using the second equality in (\ref{eq1})  we get by Lemma  \ref{Crr}:
$$\sup_{R_{r}} |\tilde \tau_n|_h\leq C_{\frac{r}{2},r}\ \forall n\geq m_{\frac{r}{2}},
$$
hence
\begin{equation}\label{jn}
\sup_{R_{r}} |\tilde j_n|_h\leq \big\{C_{\frac{r}{2},r}\big\}^{\frac{1}{2}}\ \forall n\geq m_{\frac{r}{2}}.
\end{equation}
Using different arguments one can also bound $\tilde j_n$ uniformly  on the annuli
$$\Omega^x_{\frac{r}{2},r}:=\big\{y\in X|\ \frac{r}{2}\leq \|\cg_x(y)\|\leq r\big\},\ x\in S_u.
$$
Indeed, note first that, by (\ref{jnPBs}),  $\tilde j_n$ is a bundle  morphism 
 $$\tilde j_n:\resto{F_u}{R_{\varepsilon_n}}\to \resto{F_u}{R_{\varepsilon_n}}
 $$
which is holomorphic  with respect to the pair $(\delta_u^{u_n},\delta_u'^{u_n})$. Using (iii),   fix  trivialisations   of ${\cal V}$ and ${\cal F}_u^{\we}$ on $B_r(x)$. Taking into account (\ref{delta-conv}), and (\ref{delta'-conv}),   it follows that the sequences 
$$\left(\resto{\delta^{u_n}_u}{\Omega_{\frac{r}{2},r}}\right)_n,\ \left(\resto{\delta'^{u_n}_u}{\Omega^x_{\frac{r}{2},r}}\right)_n$$
 both converge to the standard (trivial) semi-connection  of the trivial rank 2-bundle  $\Omega^x_{\frac{r}{2},r}\times \C^2$ in  ${\cal C}^\infty(\Omega^x_{\frac{r}{2},r})$. Therefore  $\resto{\tilde j_n}{\Omega_{\frac{r}{2},r}}$ is holomorphic with respect to an integrable semi-connection  which tends (as $n\to \infty$) to the trivial  semi-connection on the trivial bundle   $\Omega_{\frac{r}{2},r}\times \Hom(\C^2,\C^2)$. Therefore Corollary \ref{ann-est} (in which we take $n=2$, $k=4$, $r=\frac{1}{2}$) applies, and gives a uniform $L^\infty$-estimate of $ \tilde j_n$ on the annulus  $\Omega_{\frac{r}{2},r}$ in terms of the  supremum  of its pointwise norm on   the outer boundary of this annulus. Taking into account (\ref{jn}), we obtain the following important
\begin{lm}  There exists $m\geq m_{\frac{r}{2}}$ and a positive constant $C\geq \big\{C_{\frac{r}{2},r}\big\}^{\frac{1}{2}}$ such that
\begin{equation}\label{unif-est-j}
	\sup_{R_{\frac{r}{2}}} |\tilde j_n|_h\leq C\  \forall n\geq m.
\end{equation}
\end{lm}

Let now $\alpha>0$ be sufficiently small such that \\
\begin{enumerate}[(a)]
\item $\bar B_\Sigma(\alpha)\cap \bar B_{S_u}(r)=\emptyset$,\\
\item 	$\bar B_p(\alpha)\cap \bar B_q(\alpha)=\emptyset$  for $(p,q)\in \Sigma\times \Sigma$ with $p\ne q$,\\
\item \label{c} $\| \tilde j_n\|_{L^4(\bar B_\Sigma(\alpha))}\leq \frac{1}{2}$ for any $n\geq m_{\frac{r}{2}}$ (which can be achieved by (\ref{jn})).
\end{enumerate}
\vspace{2mm}
Consider the open set
$$T:=\cringle{R}_{\frac{r}{2}}\setminus \bar B_\Sigma(\alpha)=X\setminus \big(\bar B_{S_u}(  {r}/{2})\cup \bar B_\Sigma(\alpha)\big),
$$
and note that, by (\ref{eq1}),   and (\ref{c}) one has
\begin{equation}\label{BoundL4}
\frac{1}{2}\leq \| \tilde j_n\|_{L^4(T)}\leq 1.	
\end{equation}

The sequences  of restrictions  $(\resto{\delta_u^{u_n}}{T})_n$, $(\resto{\delta'^{u_n}_u}{T})_n$ are both convergent in the Fréchet topology ${\cal C}^\infty(T)$, and their limits are $\resto{\delta_u}{T}$, $\resto{\nu^*(\bar\partial_A)}{T}$ respectively.  Using the estimate (\ref{unif-est-j}) and a standard bootstrapping argument as in \cite[p. 469]{Morg}, one obtains, after passing to a subsequence,  a limit
$$\tilde j_\infty=\lim_{n\to \infty} \big( \resto{\tilde j_n}{T}\big), \ j_\infty:\resto{F_u}{T}\to \resto{F_u}{T}
$$
${\cal C}^\infty(T)$; this limit will be holomorphic with respect to the limit pair of semi-connections  $(\resto{\delta_u}{T},\resto{\nu^*(\bar\partial_A)}{T})$. Moreover, taking into account the lower bound in (\ref{BoundL4}), it follows that  $\tilde j_\infty\ne 0$.  Using Hartogs' Theorem, we obtain a non-trivial holomorphic map 
$$\tilde g_\infty:=\nu\circ \tilde j_\infty: {\cal F}_u^\we\to {\cal V}.$$
Taking into account the definitions of $g_n$ (formula (\ref{DefJnGn})) and $\tilde g_n$,  we have
$$\lim_{n\to\infty} \resto{\tilde g_n}{T}= \resto{\tilde g_\infty}{T}
$$
in ${\cal C}^\infty(T)$. As in \cite[p. 469]{Morg} one can prove (repeating the arguments above for  new pairs $(r',\alpha')$ as above, with  $0<r'<r$, $0<\alpha'<\alpha$) that  for any compact 4-dimensional submanifold with boundary $\Gamma\subset X\setminus (S_u\cup \Sigma)$  one still has 
$$\lim_{n\to\infty} \resto{\tilde g_n}{\Gamma}= \resto{\tilde g_\infty}{\Gamma}.
$$
This proves the claim that $\resto{\lambda_n g_n}{\Gamma}\to \resto{g_\infty}{\Gamma}\hbox{ in } {\cal C}^\infty(\Gamma)$ stated in the conclusion of the theorem. The claim
$$\resto{(\lambda_n g_n)_*(A(H^{u_n}_u,\delta_u^{u_n}))}{\Gamma}\to \resto{A}{\Gamma}$$
follows from formula (\ref{ConvConn}), and the claim (2) in the conclusion of the theorem follows from  formula (\ref{OnTheBalls}), taking into account that $\psi_{n*}(H^{u_n}_u)=H_{u_n}$, $\psi_{n*}(\delta_u^{u_n})=\delta_{u_n}$. 
\end{proof}

The remainder of the proof of Proposition \ref{PropCont2} follows almost literally that in 
sections 4.2.2 and 4.2.3 of \cite{Morg}. Briefly, the analysis splits again into
two cases, namely when ${\cal F}_u$ is stable and when it is properly
semi-stable respectively.  The only (self-evident) change required is in
the properly semi-stable case, to take account of the fact that 
$\det ({\cal F}_u)=\det({\cal F}_{u_n})={\cal D}$, which is not assumed
to be trivial in our case. Thus the destabilising sequence in the result
analogous to Theorem 4.2.11 of \cite{Morg} has the form 
$$
0\to {\cal L} \to {\cal F}_u^{\we} \to {\cal L}^{-1}\otimes {\cal D}
\otimes {\cal I}_W\to 0\;,
$$
where $2\deg_g({\cal L}) = \deg_g({\cal D})$.

\section{Appendix}

\subsection{An identity for holomorphic sections in Hermitian holomorphic bundles}

Let $(X,g)$ be a Hermitian manifold, $({\cal E},h)$ be a Hermitian holomorphic bundle on $X$. Using the convention ``a Hermitian product is linear with respect to the second argument",  a simple computation gives
\begin{lm}\label{ddbar|s|2} Let $s\in H^0({\cal E})$ and let  $A$ be the Chern connection of the  Hermitian holomorphic bundle    $({\cal E},h)$. One has the identity
\begin{equation}\label{ddc}
i \partial\bar\partial \big(|s|_h^2\big)=(-i F_{A} s,s)_h-i(\partial_{A} s\wedge \partial_{A} s)_h,
\end{equation}
where the second term 	$-i(\partial_{A_\alpha} s\wedge \partial_{A_\alpha} s)_h$ is a non-negative (1,1)-form. Moreover one has
$$-i\Lambda(\partial_{A_\alpha} s\wedge \partial_{A_\alpha} s)_h=|\partial_{A_\alpha} s|^2.
$$
In particular one has the pointwise inequality
\begin{equation}\label{ineQs}
-i \Lambda_g\partial\bar\partial \big(|s|_h^2\big)\leq 	(i\Lambda_g F_{A} s,s)_h.
\end{equation}

\end{lm}

\begin{co}\label{TrEstimateCo}
Let $(X,g)$ be a complex Hermitian manifold, $(E,h)$, $(E',h')$ be differentiable Hermitian complex bundles on $X$, let $\delta$, $\delta'$ be integrable semi-connections on $E$, $E'$, and let $f\in A^0(\Hom(E,E'))$ be a bundle morphism which is holomorphic with respect to the pair $(\delta,\delta')$. Denote by $A$, $A'$ the Chern connections associated with the pairs $(\delta,h)$, $(\delta',h')$. Then
\begin{equation}\label{TrEstimate}
-i\Lambda_g\partial\bar\partial \tr(f^*\circ f)	\leq (f\circ f^*,i\Lambda F_{A'})-(i\Lambda_g F_A,f^*\circ f).
\end{equation}
	
\end{co}
\begin{proof}
Endow the bundle $F:=E^\smvee\otimes E'$ with the integrable semi-connection $\theta:=\delta^\smvee\otimes \delta$ and with the Hermitian metric $\chi:=h^\smvee\otimes h'$. Let $B$ be the Chern connection of the pair $(\theta,\chi)$. The map $f$ can be identified with a $\theta$-holomorphic section of $F$, and one has $\tr(f^*\circ f)=| f|^2_\chi$.
It suffices to note that $F_Bf=F_{A'}\circ f-f\circ F_A$.
\end{proof}
%
%\subsection{The Daskalopoulos-Wentworth estimate}
%
%We refer to \cite{DW}.
%
%We denote by $\|\cdot\|$ the operator norm on the space of matrices. Therefore
%%
%$$\| X\|=\sup\{\|X x\|\  x\in\C^r,\ \|x\|=1\} 
%$$
%
%\begin{lm}
%Let $k\in\N^*$ and let $X_i\in M_{r,r}(\C)$ for $1\leq i\leq k$.  Then
%%
%$$\| \prod_{i=1}^k (I_r+ X_i)-I_r\|\leq e^{\sum_{i=1}^k \|X_i\|}-1.
%$$	
%\end{lm}
%\begin{proof}
%
%Put $Y_j:=\prod_{i=1}^j (I_r+ X_i)-I_r$. We prove by induction that 
%%
%\begin{equation}\label{ineq}
%\| Y_j\|\leq e^{\sum_{i=1}^j \|X_i\|}-1.
%\end{equation}
% 
%For $j=1$ the inequality reduces to $\|X_1\|\leq e^{\|X_1\|}-1$, which is obviously true. Let $j>1$ and suppose  that (\ref{ineq}) holds for $j-1$. We have
%%
%$$\|Y_j\|=\|(I_r+X_j)(I_r+Y_{j-1})-I_r\|\leq \|X_j\|+ \|Y_{j-1}\|+\|X_j\|\|Y_{j-1}\|\leq
%$$
%%
%$$\leq (e^{\sum_{i=1}^{j-1} \|X_i\|}-1)(1+\|X_j\|)+\|X_j\|=e^{\sum_{i=1}^{j-1} \|X_i\|}(1+\|X_j\|)-1\leq e^{\sum_{i=1}^j \|X_i\|}-1,
%$$
%where for the last inequality we used 
%%
%$$(1+\|X_j\|)\leq e^{\|X_j\|}.
%$$
%	
%\end{proof}
%
%We solve the equation
%%
%$$\bar\partial G+ aG=0
%$$
%
%The Leray-Koppelman operators: 
%
%$$\varphi=\bar\partial P(\varphi)+ Q(\bar \partial \varphi)
%$$
%
%
%Put $g_0=I_r$,  and define inductively
%%
%$$ g_{j+1}=I_r-P(a_j), \ a_{j+1}:=g_{j+1}^{-1}(\bar\partial g_{j+1}+ a_j g_{j+1}). $$
%
%%
%$$g_{j+1}=I_r-P(a_j)=I_r-P\big( g_j^{-1}(\bar\partial g_j+ a_{j-1} g_j)\big)
%$$
%

\subsection{Estimates for holomorphic sections on an annulus}\label{AnnEst}

Put 
$$D:=\{u\in\C^n|\ \|u\|\leq 1\},\ S:=\{u\in\C^n|\ \|u\|= 1\}.$$ 
Endow the trivial rank $k$-holomorphic bundle $E:=D\times\C^k$ over $D$ with the Hermitian $G$  obtained by multiplying the standard metric with the function   
$$u\mapsto \exp\big(\frac{\|u\|^2}{2}\big).
$$
The curvature of the corresponding Chern connection is
$$F_G=-\frac{1}{2}\sum_{j=1}^n dz_i\wedge d\bar z_i= i\omega_0\otimes\id_{E},
$$
where $\omega_0$ stands for the standard Kähler form on $\C^n$, hence $(E,G)$ is a negative Hermitian holomorphic bundle in the sense of Griffiths  (see for instance \cite[section 6]{DrF}).
For $r\in(0,1)$ put
$$\Omega_{r,1}:=\{u \in\C^n|\ r\leq \|u\|\leq 1\},
$$
and denote by $E_r$, $G_r$ the restrictions of $E$, respectively $G$ to $\Omega_{r,1}$. Let $\alpha\in A^{0,1}(\Omega_{r,1},\gl(k))$ be a (0,1)-form such that $\bar\partial\alpha+\alpha\wedge\alpha=0$, i.e. such that the semi-connection $\delta_\alpha:=\bar\partial+\alpha$ on $E_r$ is integrable, and denote by ${\cal E}_r^\alpha$ the corresponding holomorphic bundle. If $\alpha$ is sufficiently ${\cal C}^1(\Omega_{r,1})$-small, the Hermitian holomorphic bundle $({\cal E}^\alpha_r,\resto{G}{\Omega_{r,1}})$ will still be Griffiths negative. 
\begin{pr}\label{Griffiths} Suppose $n\geq 2$. There exists a positive constant $\eta_r$ such that for any 
(0,1)-form $\alpha\in A^{0,1}(\Omega_{r,1},\gl(k))$ with 
\begin{equation}\label{condi}
\bar\partial\alpha+\alpha\wedge\alpha=0,\ \|\alpha\|_{{\cal C}^1(\Omega_{r,1})}<\eta_r,	
\end{equation}
and any holomorphic section $s\in H^0(\Omega_{r,1},{\cal E}_r^\alpha)$ one has	
$$\sup_{\Omega_{r,1}} |s|_G= \sup_{S} |s|_G.
$$
\end{pr}
\begin{proof}

Let $s\in H^0(\Omega_{r,1},{\cal E}_r^\alpha)$. Using the holomorphy condition $\delta_\alpha(s)=0$ and Lemma \ref{ddbar|s|2} we obtain   the formula
\begin{equation}\label{ddcG}
i \partial\bar\partial \big(|s|_G^2\big)=(-i F_{A_\alpha} s,s)_G-i(\partial_{A_\alpha} s\wedge \partial_{A_\alpha} s)_G,
\end{equation}
where $A_\alpha$ stands for the Chern connection of the pair $({\cal E}_r^\alpha,G_r)$.
The second term on the right in (\ref{ddcG}) is obviously a non-negative (1,1)-form. Choose 	$\eta_r$ such that   $({\cal E}_r^\alpha,G_r)$ is Griffiths negative for any $\alpha\in A^{0,1}(\Omega_{r,1},\gl(k))$ satisfying (\ref{condi}). For such $\alpha$ the first term on the right in (\ref{ddcG}) will also be a non-negative (1,1)-form, hence $|s|_G^2$ will be plurisubharmonic\footnote{If $s$ is nowhere vanishing, one can even prove \cite{Be} that $\log(|s|^2_G)$ is  psh, which is a stronger property.} on $\Omega_{r,1}$. Therefore  the restriction of $|s|_G^2$ to any holomorphic disk $\Delta\subset \Omega_{r,1}$ will be subharmonic. The statement follows using the maximum principle, and taking into account that, since $n\geq 2$, any point $u\in \Omega_{r,1}$ belongs to a disk $\Delta\subset \Omega_{r,1}$ with $\partial\Delta\subset S$. 
\end{proof}

\begin{co}\label{ann-est}
In the conditions and with the notations of Proposition \ref{Griffiths} the following holds: For any Hermitian metric $h$ on $E_r$ there exists $\eta_r>0$, $C_{r,h}\geq 1$ such that for any $\alpha\in A^{0,1}(\Omega_{r,1},\gl(k))$ satisfying (\ref{condi}), and any holomorphic section $s\in H^0(\Omega_{r,1},{\cal E}_\alpha)$ one has	
$$\sup_{\Omega_{r,1}} |s|_h\leq C_{r,h} \sup_{S} |s|_h.
$$
\end{co}
\begin{proof}
Choose $C_{r,h}:= a_{r,h} b_h$,
where:
$$a_{r,h}:= \sup_{\substack{u\in \Omega_{r,1}\\ y\in E_{r,u}\setminus\{0\}}}\frac{|y|_h}{|y|_G},\ b_h:= \sup_{\substack{u\in S\\ y\in E_{r,u}\setminus\{0\}}}\frac{|y|_G}{|y|_h}=\left( \inf_{\substack{u\in S\\ y\in E_{r,u}\setminus\{0\}}}\frac{|y|_h}{|y|_G} \right)^{-1}.
$$
\end{proof}
 Corollary \ref{ann-est} gives a priori estimates for sections which are holomorphic with respect to  perturbations $\delta_\alpha=\bar\partial+\alpha$ of the trivial holomorphic structure $\bar\partial$ in the trivial bundle $E_r$ over the annulus $\Omega_r$. This result 
can be easily generalised for sections which are holomorphic with respect to  perturbations $\bar\partial_{\cal E}+\alpha$ of the holomorphic structure $\bar\partial_{\cal E}$ of an arbitrary Hermitian holomorphic bundle ${\cal E}$ on a neighbourhood of a closed strictly pseudo-convex  hypersurface in an $n$-dimensional complex manifold ($n\geq 2$). The obtained estimate will give an alternative proof of \cite[Lemma 2.2]{Bu2}, which was the inspiration source of the results in this section.

 \end{document}